\documentclass[12pt]{amsart}
\usepackage{epsfig}
\usepackage{amsfonts}
\usepackage{color}
\usepackage{graphicx}
\usepackage{amsmath}
\usepackage{amssymb}
\newtheorem{thm}{Theorem}[section]

\newtheorem{lem}[thm]{Lemma}
\newtheorem{prop}[thm]{Proposition}
\newtheorem*{prob*}{Problem}
\newtheorem*{thm*}{Theorem}

\theoremstyle{definition}
\newtheorem{defn}[thm]{Definition}

\newtheorem*{defn*}{Definition}
\newtheorem{rem}[thm]{Remark}
\newtheorem*{Remarks}{Remarks}
\newtheorem*{rem*}{Remark}
\numberwithin{equation}{section}

\DeclareMathOperator{\Prob}{Prob}

\DeclareMathOperator{\Y}{\mathbb{Y}}

\DeclareMathOperator{\Ewens}{Ewens}

\DeclareMathOperator{\Hoppe}{Hoppe}

\DeclareMathOperator{\Var}{Var}

\begin{document}
\title[A refinement of the Ewens sampling formula]
{\bf{A refinement of the Ewens sampling formula}}
\author{Eugene Strahov}
\address{Department of Mathematics, The Hebrew University of Jerusalem, Givat Ram, Jerusalem 91904, Israel}
\email{strahov@math.huji.ac.il}

\commby{}
\begin{abstract} We consider an infinitely-many neutral allelic model of population genetics where all alleles
are  divided into a finite number of classes, and each class is characterized by its own mutation rate.  For this model the allelic composition of  a  sample taken from
a very large population of genes is characterized by a random matrix, and  the   problem is to describe  the joint distribution of the matrix entries. The answer is given by
a new generalization of the classical Ewens sampling formula called the refined Ewens sampling formula in this paper. We discuss a Poisson approximation for the refined Ewens sampling formula and present its derivation by several methods. As an application,
we obtain limit theorems
for the numbers of alleles in different asymptotic regimes.

\end{abstract}

\maketitle
\section{Introduction}
\subsection{The classical Ewens sampling formula}
 In the context of population genetics,
the classical Ewens sampling formula \cite{Ewens} arises as follows. It is assumed that
a very large population of genes evolves according to the neutral Wright-Fisher model with mutations. Each individual in the population has exactly one parent, and everyone
has an equal chance of reproductive success. Mutations are
characterized by a constant mutation rate, and each mutation leads to a new allele\footnote{The paper uses a few elementary notions from
population genetics. In particular, by an allele one means an alternative form of a gene that arises by mutation. For a brief description of genetic background, we refer the reader to Feng \cite{Feng}, Chapter 1, to Durrett \cite{Durrett}, Section 1.1, and to references therein.}
not previously seen in the population (this is called the infinite allelic model). If a sample of size $n$ is taken from such a population, the probability $p\left(a_1,\ldots,a_n;\theta\right)$ that
the sample contains $a_1$ alleles represented once, $a_2$ alleles represented twice, and so on is given by the Ewens sampling formula
\begin{equation}\label{ESF}
p\left(a_1,\ldots,a_n;\theta\right)=\frac{n!}{(\theta)_n}\prod\limits_{j=1}^n
\left(\frac{\theta}{j}\right)^{a_j}\frac{1}{a_j!}.
\end{equation}
Here, $\theta>0$  is a parameter
characterizing the mutation rate,
$(\theta)_n=\theta(\theta+1)\ldots(\theta+n-1)$ is the Pochhammer
symbol, and $(a_1,\ldots,a_n)$ is an allelic partition satisfying
$\sum_{j=1}^nja_j=n$.
For details of such a derivation of the Ewens sampling formula, we refer the reader to the book by Feng \cite{Feng}, Chapters 2 and 4,
 to the book by Durrett \cite{Durrett}, Section 1.3, and to the book by Etheridge \cite{Etheridge}, Chapter 2.
For a relation of the Ewens sampling formula to different topics
of mathematics, and for different recent developments around this formula, see the articles by Tavar$\acute{\mbox{e}}$ \cite {Tavare}, by Crane \cite{Crane}, and by Griffiths and Lessard \cite{GriffithsLessard}.

\subsection{Description of the model, and the main result}
\subsubsection{Motivation}
In the derivation of the Ewens sampling formula, all mutations are assumed to occur with the same probability $\mu$. In other words, all possible alleles belong to the same class defined by $\mu$. However, it is not hard to imagine a situation where\textit{ all possible alleles can be classified into a finite number of distinct classes}.  Then we assume that the number of such classes is $k$,  and mutations to alleles from the $l$th class, $l\in \{1,\ldots,k\}$, take place with the mutation probability $\mu_l$  associated with this
specific class. Each $\mu_l$, $l=1,\ldots, k$,
is the probability for an individual to mutate into a new allele of class $l$,
$l\in\left\{1,\ldots,k\right\}$.

An example motivating the consideration of infinitely many allelic models
in which the set of all possible alleles is separated into a finite number of distinct classes arises in the study of
Cystic Fibrosis (CF). This is a disease that affects approximately 28,000 patients
in the United States and approximately 36000 patients in Europe, see
O'Reilly and Elphick \cite{ReillyElphick}, and references therein.
Cystic fibrosis is known to be caused by mutations of a specific gene.
This gene identified in 1989 encodes the CF transmembrane conductance regulator
(CFTR) protein and has more than 1900 mutations with the potential to cause
disease. Mutations can be classified into six different classes, according to their
consequences on CFTR properties; see Table \ref{tab1}. Each class
contains hundreds of different alleles and is characterized by its own mutation rate.

\begin{table}
    \centering
    \begin{tabular}{cc}
     \textbf{Class}    &  \textbf{Consequence of mutation}\\
     \\
     \hline
     \\
      1   & No CFTR synthesis \\
      \\
      2   & Failure of CFTR to traffic to the cell surface \\
      \\
      3   &  Failure of CFTR to open at cell surface\\
      \\
      4   &  Decreased time CFTR channel is open at cell surface\\
      \\
      5   &  Decreased amount of CFTR at cell surface\\
      \\
      6   &  Unstable CFTR with decreased half-line at the cell surface\\
      \\
      \hline
    \end{tabular}
    $$
    $$
    \caption{ Six different classes of CFTR mutations. The table is taken from O'Reilly and Elphick \cite{ReillyElphick}.}
    \label{tab1}
\end{table}

\subsubsection{The model}\label{SectionModelDefinition}
The model under considerations is a Wright-Fisher-type model with infinitely many allelic types
and mutations.  We consider a population
of $2N$ genes whose size is preserved in each generation $t$, $t\in\left\{0,1,2,\ldots\right\}$. Time $t=0$ refers to the starting generation,
and $t=m$ corresponds to the generation $m$. The population evolves under the influence
of mutation and random sampling with replacement (as in the usual Wright-Fisher models with mutations; see, for example, Ewens \cite{EwensLectures}). Each mutation gives rise to a new allele previously not seen in the population, and all
possible alleles form a countably infinite set $\mathcal{A}$. In addition, in contrast to the usual Wright-Fisher infinitely many allele model described, for example,
in Ewens \cite{EwensLectures}, Etheridge \cite{Etheridge}, Feng \cite{Feng}
we assume that $\mathcal{A}$ is a union of $k$ countably infinite non-intersecting
subsets $\mathcal{A}^{(l)}$, i.e. $\mathcal{A}=\cup_{l=1}^k\mathcal{A}^{(l)}$.
Alleles of class $l$ are the elements of $\mathcal{A}^{(l)}$, and the mutation of each existing gene to each allele of class $\mathcal{A}^{(l)}$ is possible
with the probability of mutation $\mu_l$.

The reproduction and mutation mechanisms are defined as follows.
In each generation $t=1,2,\ldots$, each gene randomly selects a parent from the previous generation uniformly. The gene inherits the allelic type of the parent
with probability $1-\sum_{l=1}^k\mu_l$, or mutates to a new allelic type
in the class $\mathcal{A}^{(l)}$ with probability $\mu_l$.

We assume that the process described above has been going on for such a long time that it reaches statistical equilibrium.
\subsubsection{Main result}
Suppose we examine a sample of $n$ gametes taken from such a population. We will be interested in the matrix
\begin{equation}\label{Data}
\left(
\begin{array}{ccc}
  a_1^{(1)} & \ldots & a_1^{(k)} \\
  \vdots &  &  \\
  a_n^{(1)} & \ldots & a_n^{(k)}
\end{array}
\right),\;\;\sum\limits_{l=1}^k\sum\limits_{j=1}^nja_j^{(l)}=n,
\end{equation}
where $a_j^{(l)}$ is\textit{ the number of alleles of class $l$ represented $j$ times in the sample}. Such a measurement would be possible if the experimental techniques permit labeling of the alleles we can distinguish according to their classes only.  We say that matrix (\ref{Data}) describes the allelic composition of a sample.

If a random sample of size $n$
is taken from the population, the corresponding data is a \textit{random matrix}
$A_N(n)$ of the same format as (\ref{Data}), i.e.
$$
A_N(n)=
\left(
\begin{array}{ccc}
  A_{N,1}^{(1)}(n) & \ldots & A_{N,1}^{(k)}(n) \\
  \vdots &  &  \\
  A_{N,n}^{(1)}(n) & \ldots & A_{N,n}^{(k)}(n)
\end{array}
\right),\;\;\sum\limits_{l=1}^k\sum\limits_{j=1}^njA_{N,j}^{(l)}(n)=n,
$$
where $A_{N,j}^{(l)}(n)$ is a random variable characterizing the number of alleles of class $l$ represented $j$ times in the sample of size $n$.
The present paper initiates a study of such random matrices, and our result
is the following theorem.
\begin{thm}\label{MAINTHEOREM}
Assume that $\mu_1=\frac{\theta_1}{4N}$, $\ldots$, $\mu_k=\frac{\theta_k}{4N}$, where $\theta_1>0$, $\ldots$, $\theta_k>0$ are fixed numbers
and $2N$ is the size of the population.
As $N\longrightarrow\infty$, the random variable $A_{N,j}^{(l)}(n)$ converges
in distribution to a random variable $A_{j}^{(l)}(n)$ (for each $j\in\left\{1,\ldots,n\right\}$ and each $l\in\{1,\ldots,k\}$). The joint probability distribution
of $A_j^{(l)}(n)$ is given by
\begin{equation}\label{RefinedESF}
\begin{split}
&\Prob\left(A_j^{(l)}(n)=a_j^{(l)};\; l=1,\ldots,k;\;j=1,\ldots,n\right)\\
&=
\mathbf{1}\left(\sum\limits_{l=1}^k\sum\limits_{j=1}^nja_j^{(l)}=n\right)
\frac{n!}{\left(\theta_1+\ldots+\theta_k\right)_n}\prod\limits_{l=1}^k\prod\limits_{j=1}^n
\left(\frac{\theta_l}{j}\right)^{a_j^{(l)}}\frac{1}{a_j^{(l)}!},
\end{split}
\end{equation}
where $\{a_j^{(l)}:\;l=1,\ldots,k;\; j=1,\ldots,n \}$ are non-negative integers, and $(\theta)_n=\theta (\theta+1)\ldots(\theta+n-1)$ is the Pochhammer symbol.
\end{thm}

\begin{Remarks}
a) If we consider the example that arises in the study of Cystic Fibrosis (see  Table \ref{tab1}), then $k=6$. Assume that the mutations of the first class (without CFTR synthesis) are characterized
by the mutation parameter $\theta_1$, the mutations of the second class (with failure of CFTR
to traffic to the cell surface) are characterized by the mutation parameter $\theta_2$, and so on.
Then the probability that all genes in a sample of size $n$ are of the same allele, and the mutation affects the biosynthesis of CFTR (i.e. the allele is of the first class) is
$$
\frac{(n-1)!\theta_1}{\left(\theta_1+\theta_2+\theta_3+\theta_4+\theta_5+\theta_6\right)_n},
$$
as follows immediately from formula (\ref{RefinedESF}).

b) For $k=1$ (which means that there is no distinction between alleles from different classes), formula (\ref{RefinedESF}) turns into the Ewens sampling formula, and we refer to (\ref{RefinedESF}) as to the \textit{refined Ewens sampling formula}.

c) Note that different extensions of the Ewens sampling formula (equation (\ref{ESF})) are known.
In particular, the Pitman sampling formula (which first appeared in Pitman \cite{Pitman1} and was
further studied in Pitman \cite{Pitman2}) can be understood as a two-parameter generalization of
formula (\ref{ESF}). The paper by Griffiths and Lessard \cite{GriffithsLessard} extends the Ewens sampling formula and the related Poisson-Dirichlet distribution to a model where the population size may vary over time. Another family of parametric sampling formulas is introduced in Gnedin \cite{Gnedin}. However, a generalization of the Ewens sampling formula to a model where the set of all possible allelic types is partitioned into $k$ infinite countable classes is not available in the existing literature, and formula (\ref{RefinedESF}) is a new result,
to the best of my knowledge.
\end{Remarks}
\subsection{Organization of the paper, and summary of results}
The paper is organized as follows. In Section \ref{SectionRESFMPS}
we prove that equation (\ref{RefinedESF}) indeed defines a probability distribution on the set of non-negative integers; see Proposition
\ref{PropositionIndeedDistribution}. To prove Proposition \ref{PropositionIndeedDistribution} we interpret each matrix (\ref{Data}) as a multiple partition $\Lambda_n^{(k)}$  of $n$
into $k$ components. Then we show that equation (\ref{RefinedESF})  defines a probability measure
$M_{\theta_1,\ldots,\theta_k; n}^{\Ewens}$ on the set $\Y_n^{(k)}$
of all such multiple partitions $\Lambda_n^{(k)}$. Also, we show that
expressions (\ref{RefinedESF}) are consistent for different values of $n$: if from a sample of size $n$ whose allelic composition has distribution (\ref{RefinedESF}) a random sub-sample of size $n-1$ is
taken without replacement, then the distribution of the allelic composition of the sub-sample is given by (\ref{RefinedESF})
with $n$ replaced by $n-1$. This property of (\ref{RefinedESF})
can be equivalently reformulated as a condition on the system
$\left(M_{\theta_1,\ldots,\theta_k;n}^{\Ewens}\right)_{n=1}^{\infty}$.
Namely, $\left(M_{\theta_1,\ldots,\theta_k;n}^{\Ewens}\right)_{n=1}^{\infty}$
should be a \textit{multiple partition structure}, see Definition \ref{DefinitionMultiplePartitionStructure}. We prove this property
of $\left(M_{\theta_1,\ldots,\theta_k;n}^{\Ewens}\right)_{n=1}^{\infty}$
in Proposition \ref{PropositionMEWENSMPS}, and thus the consistency
of (\ref{RefinedESF}) for different values of $n$ follows.

In this paper we derive the refined Ewens sampling formula as that for the stationary distribution (in the infinite population limit)
of a Wright-Fisher type model with mutations. Each mutation leads to a new allele not previously seen in the population (as in the usual infinite-allele model). The new assumption here is that all possible alleles are divided into $k$ different classes, and each class is characterized by its own mutation rate. We present two approaches to the refined Ewens sampling formula.
Our first approach (see Section \ref{SectionProof}) works backward in time: we observe that the model under considerations leads to the genealogical process which is the coalescent with killing, the killing is characterized by $k$ different rates (depending on the allelic class).
In the infinite population limit, the genealogical process is associated with a \textit{generalization of the Hoppe urn model} \cite{Hoppe1,Hoppe2}. The relation with
the generalized Hoppe urn model enables us to compute the stationary distribution, and thus to derive the refined Ewens sampling formula,
equation (\ref{RefinedESF}). This proves Theorem \ref{MAINTHEOREM}.

The second derivation (see Section \ref{SectionDerivation1}) works forward in time:
we consider an infinite population divided into $k$ classes, where each class is a union of a countable number of sub-populations characterized by their own
allelic types. Let $x_j^{(l)}$ denote the proportion of the population characterized by the $j$th allelic type of class $l$. Assuming that
$x_j^{(l)}$ are fixed numbers, we compute the probability
$\mathbb{K}_n\left(\Lambda_n^{(k)};x^{(1)},\ldots,x^{(k)}\right)$ that the alleles presented in the sample are characterized by a multiple partition $\Lambda_n^{(k)}$. Then, we allow $x_j^{(l)}$ to be random.
Under natural assumptions, we obtain that the joint distribution of $x_j^{(l)}$ is the \textit{multiple Poisson-Dirichlet distribution}
$PD\left(\theta_1,\ldots,\theta_k\right)$.
Then the probability that the alleles presented in the sample are characterized by $\Lambda_n^{(k)}$ is obtained by integration of
$\mathbb{K}_n\left(\Lambda_n^{(k)};x^{(1)},\ldots,x^{(k)}\right)$
over certain set $\overline{\nabla}_0^{(k)}$, and $PD\left(\theta_1,\ldots,\theta_k\right)$ is the integration measure
on this set. The result of this integration is the value of
$M_{\theta_1,\ldots,\theta_k; n}^{\Ewens}$ at $\Lambda_n^{(k)}$, which is equivalent to the refined Ewens sampling formula; see Proposition
\ref{PropositionMEwenSASIntegral}.

The interpretation of $\theta_1$, $\ldots$, $\theta_k$ as parameters describing mutations in alleles of different classes dictates a consistency condition on
$M_{\theta_1,\ldots,\theta_k;n}^{\Ewens}$.  Assume that $M_{\theta_1,\ldots,\theta_k;n}^{\Ewens}$ defines the distribution of a random variable describing the allelic composition of a sample of size $n$.
If experimental techniques do not enable distinguishing between $k$ different classes of alleles, then the allelic composition of a sample
should be described by a random Young diagram with $n$ boxes distributed according to the usual Ewens formula with the parameter $\theta_1+\ldots+\theta_k$. This is shown explicitly in Section \ref{SectionConsistency}.

In Section \ref{SectionChineseMain}, we relate the refined Ewens sampling formula with a family of probability distributions on wreath products. Namely,
in Section \ref{SectionChinese} we show that the refined Ewens sampling formula can be understood as a probability distribution
$P_{t_1,\ldots,t_k;n}^{\Ewens}$ on a wreath product
$G\sim S(n)$ of a finite group $G$ with the symmetric group $S(n)$, see Proposition \ref{PropositionCorMeasures}. For wreath products $G\sim S(n)$, we construct an \textit{ analog of the Chinese
restaurant process} \cite{Aldous} that generates $P_{t_1,\ldots,t_k;n}^{\Ewens}$, see Section \ref{SectionChinese2}.

In Section \ref{SectionRandomSetPartitions}
we discuss random set partitions associated with
the refined Ewens Sampling Formula model.
In Sections \ref{SectionNumbersAllelicTypes} and
\ref{SectionRepresentationPoisson} we present several applications of the refined Ewens sampling formula.
In Section \ref{SectionNumbersAllelicTypes} we study the numbers $K_n^{(l)}$ of alleles of class $l$
in a random sample of size $n$. These numbers can be experimentally observed and, therefore, $K_n^{(l)}$
are important probabilistic quantities of interest. In the context of the example arising in the study of Cystic Fibrosis  $K_n^{(1)}$ represents the number of mutations in a random sample of size
$n$ that affect biosynthesis, $K_n^{(2)}$ is the number of mutations that affect the traffic to the cell surface, and so on.
We obtain exact formulas for the expectation of $K_n^{(l)}$,
for the variance of $K_n^{(l)}$, and for the joint distribution of $K_n^{(1)}$, $\ldots$, $K_n^{(k)}$. Then we investigate the asymptotic properties of these random variables as $n\rightarrow\infty$, in the situation
where the mutation parameters increase together
with the sample size $n$.
Our motivation to consider the regimes where both $n$ and the mutation parameters tend to infinity comes from the  following observations. First, it is very common to have a large sample size in actual experiments. In fact, Tavar$\acute{\mbox{e}}$ \cite{Tavare} presents two data sets describing sample sizes and allele frequencies
for fruit fly species \textit{Drosophila tropicalis} and \textit{Drosophila simulans}. The sample size for Drosophila tropicalis is equal to $298$, and for Drosophila simulans it is equal to $308$.
 In studies to detect causal genes of complex human diseases, the required sample size varies from several hundreds to several thousands, see, for example, Hong and Park \cite{HongPark}.
Second, the mutation parameters $\theta_1$, $\ldots$, $\theta_k$ are proportional to the population size $2N$, so the large mutation parameters $\theta_1$, \ldots, $\theta_k$
correspond to the large population size $2N$ with fixed $\mu_1$, $\ldots$, $\mu_k$. If, for example,
the parameter $n$ is not quite larger than $\theta_1$, $\ldots$, $\theta_k$, the asymptotic regime where
only $n$ tends to infinity does not provide a good approximation. This issue is addressed
by considering the case where both the sample size $n$ and the mutation parameters $\theta_1$,
$\ldots$, $\theta_k$ tend to infinity.
In this paper, we discuss the convergence
in probability and show asymptotic normality of the random variables $K_n^{(l)}$ in different asymptotic regimes; see Theorem \ref{TheoremKconstant}, Theorem \ref{TheoremConvergenceInProbability}, and
Theorem \ref{Theorem12.8}. In the context of the classical Ewens formula,
similar asymptotic regimes were investigated in Feng \cite{FengL} (see Section 4),  and in Tsukuda \cite{Tsukuda, Tsukuda1}. Similar results are available in the literature for the number of distinct values in a random sample from the Dirichlet process; see Antoniak \cite{Antoniak}, Korwar and Hollander \cite{KorwarHollander}, and the book by Ghosal and van der Vaart \cite{GhosalvanderVaart},
Section 4.1.5 for a detailed account and references.

As mentioned above, a random sample can be described by a random matrix $A(n)$ whose entries are random variables with joint distribution (\ref{RefinedESF}). In Section \ref{SectionRepresentationPoisson} we show that as $n\rightarrow\infty$ the matrix $A(n)$ can be approximated by a random matrix whose entries are independent Poisson random variables; see Theorem
\ref{TheoremConvergencePoisson}. Theorem \ref{TheoremConvergencePoisson} is an extension
of the well-known result by Arratia, Barbour, and Tavar$\acute{\mbox{e}}$ \cite{ArratiaBarbourTavarePoisson}.

\textbf{Acknowledgements.} I thank Natalia Strahov for bringing to my attention the paper by O'Reilly and Elphick \cite{ReillyElphick}. I am very grateful to A. Gnedin for discussion.
\section{The refined Ewens sampling formula and multiple partition structures}\label{SectionRESFMPS}
It is known that the Ewens sampling formula (equation (\ref{ESF})) can be understood as a distribution on the set $\Y_n$
of Young diagrams with $n$ boxes (or, alternatively, on the set of partitions of $n$). In fact, interpret
$a_1$ in equation (\ref{ESF}) as the number of rows with one box in a Young diagram $\lambda$, $a_2$ as the number
of rows with two boxes in $\lambda$, and so on. We obtain a correspondence
$$
\left(a_1,\ldots,a_n\right)\longrightarrow \lambda=\left(1^{a_1}2^{a_2}\ldots n^{a_n}\right),
$$
and condition $\sum_{j=1}^nja_j=n$ ensures $\lambda\in\Y_n$. Conversely, each Young diagram $\lambda$, $\lambda\in\Y_n$, gives rise (by the same rule) to a list $(a_1,\ldots,a_n)$ of positive integers such that  condition $\sum_{j=1}^nja_j=n$ is satisfied.

In the following, we explain that formula (\ref{RefinedESF}) can be understood as a distribution on the set $\Y_n^{(k)}$ of all multiple
partitions $\Lambda_n^{(k)}$ of $n$ into $k$ components. By a \textit{multiple partition of $n$ into $k$ components} we mean a family $\Lambda_n^{(k)}=\left(\lambda^{(1)},\ldots,\lambda^{(k)}\right)$ of Young diagrams
$\lambda^{(1)}$, $\ldots$, $\lambda^{(k)}$ such that the condition
\begin{equation}\label{normRESF1}
|\lambda^{(1)}|+\ldots+|\lambda^{(k)}|=n
\end{equation}
is satisfied. Here, $|\lambda^{(l)}|$ denotes the number of boxes in the Young diagram $\lambda^{(l)}$. Each matrix (\ref{Data}) that describes the allelic composition of a sample of size $n$ can be uniquely identified with
a multiple partition $\Lambda_n^{(k)}$ of $n$ into $k$ components. The
correspondence is defined as follows
\begin{equation}\label{DataPartition}
\left(
\begin{array}{ccc}
  a_1^{(1)} & \ldots & a_1^{(k)} \\
  \vdots &  &  \\
  a_n^{(1)} & \ldots & a_n^{(k)}
\end{array}
\right)\longrightarrow \Lambda_n^{(k)}=\left(\lambda^{(1)},\ldots,\lambda^{(k)}\right),
\end{equation}
where
\begin{equation}\label{CorrespondenceResf}
\lambda^{(l)}=\left(1^{a_1^{(l)}}2^{a_2^{(l)}}\ldots n^{a_n^{(l)}}\right),\; l=1,\ldots,k.
\end{equation}
Equation (\ref{CorrespondenceResf}) says that exactly $a_j^{(l)}$ of the rows of $\lambda^{(l)}$ are equal to $j$. It is easy to see that (\ref{normRESF1}) is satisfied provided that
\begin{equation}\label{ConditionMatrix}
\sum\limits_{l=1}^{k}\sum\limits_{j=1}^nja_j^{(l)}=n.
\end{equation}
Thus, $\Lambda_n^{(k)}=\left(\lambda^{(1)},\ldots,\lambda^{(k)}\right)$ (with $\lambda^{(1)}$,
$\ldots$, $\lambda^{(k)}$ defined by equation (\ref{CorrespondenceResf})) is an element of $\Y_n^{(k)}$. Conversely, each $\Lambda_n^{(k)}$  gives rise
to a matrix (\ref{Data}) describing the allelic composition of a sample,
and we say that $\Lambda_n^{(k)}$ defines \textit{the allelic composition of a sample}.

Assume that each $\Lambda_n^{(k)}\in\Y_n^{(k)}$ is defined by equations (\ref{DataPartition}) and (\ref{CorrespondenceResf}) in terms of nonnegative integers $a_j^{(l)}$ such that condition (\ref{ConditionMatrix}) is satisfied. Equation (\ref{RefinedESF}) can be rewritten as
\begin{equation}\label{NEW2.5}
\Prob\left(A_j^{(l)}(n)=a_j^{(l)};\; l=1,\ldots,k;\; j=1,\ldots,n\right)
=M_{\theta_1,\ldots,\theta_k;n}^{\Ewens}(\Lambda_n^{(k)}),
\end{equation}
where
\begin{equation}\label{MEM}
M_{\theta_1,\ldots,\theta_k;n}^{\Ewens}(\Lambda_n^{(k)})=
\frac{n!}{\left(\theta_1+\ldots+\theta_k\right)_n}\prod\limits_{l=1}^k\prod\limits_{j=1}^n
\left(\frac{\theta_l}{j}\right)^{a_j^{(l)}}\frac{1}{a_j^{(l)}!}.
\end{equation}
To rewrite (\ref{RefinedESF}) as (\ref{NEW2.5})  use equations (\ref{DataPartition})-(\ref{ConditionMatrix}).
\begin{prop}\label{PropositionIndeedDistribution}
We have
$$
\sum\limits_{\Lambda_n^{(k)}\in\Y_n^{(k)}}M_{\theta_1,\ldots,\theta_k;n}^{\Ewens}(\Lambda_n^{(k)})=1.
$$
Thus equation (\ref{RefinedESF}) indeed defines a probability distribution on the set
$$
\left\{a_j^{(l)}:\;l=1,\ldots,k;\; j=1,\ldots,n\right\}
$$
of non-negative integers.
\end{prop}
\begin{proof}
From (\ref{MEM}) we obtain that
\begin{equation}\label{MultipleEwensMeasure}
\begin{split}
&M_{\theta_1,\ldots,\theta_k;n}^{\Ewens}\left(\Lambda_n^{(k)}\right)
=\frac{\left(\theta_1\right)_{|\lambda^{(1)}|}\ldots \left(\theta_k\right)_{|\lambda^{(k)}|
}}{\left(\theta_1+\ldots+\theta_k\right)_n}\\
&\times\frac{n!}{|\lambda^{(1)}|!\ldots |\lambda^{(k)}|!}\;
M_{\theta_1;|\lambda^{(1)}|}^{\Ewens}\left(\lambda^{(1)}\right)
\ldots M_{\theta_k;|\lambda^{(k)}|}^{\Ewens}\left(\lambda^{(k)}\right),\;\;
\Lambda_n^{(k)}=\left(\lambda^{(1)},\ldots,\lambda^{(k)}\right).
\end{split}
\end{equation}
Here
\begin{equation}\label{EwensMeasure}
M_{\theta; n}^{\Ewens}\left(\lambda\right)=\frac{n!}{\prod_{j=1}^{n}j^{m_j(\lambda)}m_j(\lambda)!}\frac{\theta^{l(\lambda)}}{(\theta)_n},
\end{equation}
with $\lambda=\left(1^{m_1(\lambda)}2^{m_2(\lambda)}\ldots\right)$ is the Ewens probability measure on $\Y_n$, the set of all Young diagrams with $n$ boxes. In equation (\ref{EwensMeasure})
$l(\lambda)$ denotes the number of rows in $\lambda$.

Next we use the formula
\begin{equation}\label{SummationHyper}
n!\sum\limits_{m_1+\ldots+m_k=n}\frac{(\theta_1)_{m_1}\ldots(\theta_k)_{m_k}}{m_1!\ldots m_k!}=\left(\theta_1+\ldots+\theta_k\right)_n,
\end{equation}
and the fact that each $M_{\theta_l,m_l}^{\Ewens}$ is a probability measure on $\Y_{m_l}$.
\end{proof}

A  \textit{random multiple partition of $n$ into $k$ components} is a random variable $\mathfrak{L}_n^{(k)}$
taking values in the set $\Y_n^{(k)}$. If a sample of $n$ gametes is chosen at random from a population, then the allelic composition of a sample is described by the random variable $\mathfrak{L}_n^{(k)}$.
\begin{defn}\label{DefinitionMultiplePartitionStructure} A multiple partition structure is a sequence $\mathcal{M}_1^{(k)}$, $\mathcal{M}^{(k)}_2$, $\ldots$ of distributions for $\mathfrak{L}_1^{(k)}$, $\mathfrak{L}_2^{(k)}$,
$\ldots$ which is consistent in the following sense: if the allelic composition of a sample of $n$ gametes is $\mathfrak{L}_n^{(k)}$, and a random sub-sample of size $n-1$ is taken independently of
$\mathfrak{L}_n^{(k)}$, then the allelic composition $\mathfrak{L}_{n-1}^{(k)}$ is distributed according to
$\mathcal{M}_{n-1}^{(k)}$.
\end{defn}
Multiple partition structures were first introduced in Strahov
\cite{StrahovMPS}. Multiple partition structures can be viewed as
generalizations of Kingman partition structures \cite{Kingman1, Kingman2}.
\begin{prop}\label{PropositionMEWENSMPS}The system $\left(M_{\theta_1,\ldots,\theta_k;n}^{\Ewens}\right)_{n=1}^{\infty}$ is a multiple partition structure.
\end{prop}
\begin{proof} It was proved in Strahov \cite{StrahovMPS}, Section 8.1
that a sequence $\mathcal{M}_1^{(k)}$, $\mathcal{M}^{(k)}_2$, $\ldots$ of distributions for $\mathfrak{L}_1^{(k)}$, $\mathfrak{L}_2^{(k)}$,
$\ldots$
 is a multiple partition structure if and only if the consistency condition
\begin{equation}\label{Transition}
\mathcal{M}_{n-1}^{(k)}\left(\Lambda_{n-1}^{(k)}\right)=\sum\limits_{\Lambda_n^{(k)}\in\Y_n^{(k)}}
\Prob\left(\mathfrak{L}_{n-1}^{(k)}=\Lambda_{n-1}^{(k)}\biggl\vert\mathfrak{L}_{n}^{(k)}=\Lambda_n^{(k)}\right)\mathcal{M}_n^{(k)}\left(\Lambda_n^{(k)}\right),
\end{equation}
is satisfied for each $\Lambda_{n-1}^{(k)}$, $\Lambda_{n-1}^{(k)}\in\Y_{n-1}^{(k)}$.
Here the transition probability
to get
$\Lambda_{n-1}^{(k)}=\left(\mu^{(1)},\ldots,\mu^{(k)}\right)$ from $\Lambda_n^{(k)}=\left(\lambda^{(1)},\ldots,\lambda^{(k)}\right)$
is defined by
\begin{equation}\label{LAMBDALAMBDA}
\Prob\left(\mathfrak{L}_{n-1}^{(k)}=\Lambda_{n-1}^{(k)}\biggl\vert\mathfrak{L}_{n}^{(k)}=\Lambda_n^{(k)}\right)
=\left\{
  \begin{array}{ll}
    \frac{1}{n}m_{L^{(1)}}(\lambda^{(1)})L^{(1)}, & \lambda^{(1)}\searrow\mu^{(1)}, \\
     \frac{1}{n}m_{L^{(2)}}(\lambda^{(2)})L^{(2)}, & \lambda^{(2)}\searrow\mu^{(2)}, \\
    \vdots, &  \\
     \frac{1}{n}m_{L^{(k)}}(\lambda^{(k)})L^{(k)}, & \lambda^{(k)}\searrow\mu^{(k)}, \\
    0, & \mbox{otherwise},
  \end{array}
\right.
\end{equation}
where the notation $\lambda^{(l)}\searrow\mu^{(l)}$ means that $\mu^{(l)}$ is obtained from $\lambda^{(l)}$ by extracting a box, $L^{(l)}$ is the length of the row in $\lambda^{(l)}$
from which a box is removed to get $\mu^{(l)}$, and $m_{L^{(l)}}\left(\lambda^{(l)}\right)$ is the number of rows of size $L^{(l)}$ in $\lambda^{(l)}$.

In order to see that the system $\left(M_{\theta_1,\ldots,\theta_k;n}^{\Ewens}\right)_{n=1}^{\infty}$ is a multiple partition structure we first check directly using (\ref{EwensMeasure}) that
\begin{equation}\label{EwensZ}
M_{\theta,n-1}^{\Ewens}(\mu)=\sum\limits_{\lambda\searrow\mu}
\left(\frac{1}{n}m_L(\lambda)L\right)M_{\theta,n}^{\Ewens}(\lambda).
\end{equation}
The fact that the system $\left(M_{\theta_1,\ldots,\theta_k;n}^{\Ewens}\right)_{n=1}^{\infty}$ satisfies (\ref{Transition}) follows from equations (\ref{MultipleEwensMeasure}), (\ref{LAMBDALAMBDA})
and (\ref{EwensZ}) by direct calculations.
\end{proof}
Thus, $\left(M_{\theta_1,\ldots,\theta_k;n}^{\Ewens}\right)_{n=1}^{\infty}$
is a multiple partition structure in the sense of Definition
\ref{DefinitionMultiplePartitionStructure}. In what follows, we refer
to $\left(M_{\theta_1,\ldots,\theta_k;n}^{\Ewens}\right)_{n=1}^{\infty}$
as to the \textit{Ewens multiple partition structure.}

\section{Proof of Theorem \ref{MAINTHEOREM}}\label{SectionProof}
\subsection{Main steps of the proof}
Our proof of Theorem \ref{MAINTHEOREM} works backward in time: It is based on the coalescent with killing, leading to a generalized Hoppe's urn model.  Our arguments here are generalizations of those of Feng \cite{Feng},
Chapter 4, of Durrett \cite{Durrett}, Section 1.3, and of Tavar$\acute{\mbox{e}}$ \cite{TavareBook}.
To begin with, let us describe in words the main steps of the proof of Theorem \ref{MAINTHEOREM}. The details are presented in Sections
\ref{Section3.1}-\ref{Section3.2} below.
\subsubsection{Step 1} We start from the model defined in Section
\ref{SectionModelDefinition}. For this model,
in Section \ref{Section3.1} we construct a Markov chain
$\left\{Z_{n,2N}(t):\; t=0,1,2,\ldots\right\}$,  where
\begin{itemize}
  \item $n$ is the size of a sample taken from the population of $2N$
  genes at time $t=0$ (present time);
  \item the time $t=j$ is that of going back $j$ generations to the past;
  \item $Z_{n,2N}(0)=n$, $Z_{n,2N}(1)$ is the total  number of ancestors,
  one generation back, $Z_{n,2N}(2)$ is the total number of ancestors, two generations back, etc.
\end{itemize}
The transition probability $P_{2N}(p,m)$ of $\left\{Z_{n,2N}(t):\; t=0,1,2,\ldots\right\}$
is computed in Proposition
\ref{PropositionRelationPPo}, and Proposition \ref{Proposition5.6.1}
gives the asymptotics of $P_{2N}(p,m)$ as $N\longrightarrow\infty$.
\subsubsection{Step 2}We show that if we rescale the time so that it is measured in units of $N$ generations, and let $N\longrightarrow\infty$, then the Markov chain $\left\{Z_{n,2N}(t):\; t=0,1,2,\ldots\right\}$
is approximated by a pure death process $D_n(t)$ called the ancestral process with mutations for a sample of size $n$. This is done in Section
\ref{SectionAncestral}. In particular, we explicitly compute
the intensity of this process. This computation enables us to interpret
the ancestral process with mutations as the coalescence with killing.
\subsubsection{Step 3} The important observation is that the genealogical relationship between $n$ particles described by the ancestral process with mutations can be simulated by running the Hoppe-like urn model of Section \ref{Section3.2}  for $n$ time steps.
In fact, it is clear from the description of the generalized Hoppe urn process in Section \ref{Section3.2} that the choice of a black object corresponds to a new mutation and the choice of a colored (nonblack) object corresponds to a coalescent effect. As we go backward in the generalized Hoppe urn process from time
$j$ to time $j-1$ we lose a particle with probability given by equation
(\ref{MutationProbability}) and have a coalescence with probability given by equation (\ref{CoalescentProbability}).
In Theorem \ref{TheoremGeneralizedHoope} we compute the distribution of a multiple partition $\mathcal{L}_n^{(k)}$ associated with the Hoppe-like
urn containing objects of $k$ different classes (as described in Section \ref{Section3.2}) at time $n$. The same multiple partition describes the genealogical relationship for a random sample of size $n$ taken from a very large population described in Section \ref{SectionModelDefinition}. Therefore, we interpret $\mathcal{L}_n^{(k)}$ from Theorem \ref{TheoremGeneralizedHoope} as a random matrix describing the allelic composition of the sample and obtain the formula (\ref{RefinedESF}).

\subsection{A Markov chain for the Wright-Fisher-type model}
\label{Section3.1}
Consider the model defined in Section \ref{SectionModelDefinition}.
Choose $p$ individuals from the current generation. Among these individuals,
there are those inherited allelic types of the parents and those mutated to new allelic types; see the description of the reproduction and mutation mechanisms in Section \ref{SectionModelDefinition}. Denote by $\mathcal{E}(p,m)$ the event ``between $p$ individuals, those who inherited allele types have $m$ parents in the previous generation".
Denote by $P_{2N}(p,m)$ the probability of $\mathcal{E}(p,m)$, and by $P_{2N}^{o}(p,m)$
the probability of the same event for the same model without mutations.
There is a simple formula for
$P^{o}_{2N}(p,m)$, that is,
\begin{equation}\label{n2.2.1}
P^{o}_{2N}(p,m)=S(p,m)\frac{(2N)(2N-1)\ldots (2N-m+1)}{(2N)^p},
\end{equation}
where $S(p,m)$  is the total number of ways of partitioning $\{1,\ldots,p\}$ into $m$ non-empty subsets. To see that (\ref{n2.2.1}) is true, observe that each individual can have $2N$ parents, and $(2N)^p$ is the number of possibilities to assign $p$ individuals to their parents.
In addition, $S(p,m)$ can be understood as the number of ways to partition
$p$ individuals into $m$ groups, each group having the same parent. The number $(2N)(2N-1)\ldots (2N-m+1)$ is equal to the number of ways to choose
$m$ different parents from $2N$ possibilities.
\begin{prop}\label{PropositionRelationPPo}The relation between probabilities $P_{2N}(p,m)$ and $P_{2N}^{o}(p,m)$ is
\begin{equation}\label{n2.2}
P_{2N}(p,m)=\sum\limits_{j=0}^{p-m}\left(\begin{array}{c}
  p \\
  j
\end{array}\right)\left(1-\sum\limits_{l=1}^k\mu_l\right)^{p-j}
\left(\sum\limits_{l=1}^k\mu_l\right)^jP_{2N}^{o}(p-j,m).
\end{equation}
Here $m\in\left\{0,1,\ldots,p\right\}$, and $P_{2N}^{o}(p,0)=\delta_{p,0}$.
\end{prop}
\begin{proof}
Denote by $\mathcal{M}(j,p)$
the event that ``between $p$ individuals, there are $j$ who mutated to new allelic types''. Note that ``between $p$ individuals, there are $p-j$ who inherited allelic types of the parents" is the same event as $\mathcal{M}(j,p)$.
We have
\begin{equation}
P_{2N}(p,m)=\sum\limits_{j=0}^{p-m}\Prob\left(\mathcal{E}(p,m)|
\mathcal{M}(j,p)\right)\Prob\left(\mathcal{M}(j,p)\right).
\end{equation}
Clearly,
$$
\Prob\left(\mathcal{E}(p,m)|
\mathcal{M}(j,p)\right)=P_{2N}^{o}(p-j,m),
$$
and
$$
\Prob\left(\mathcal{M}(j,p)\right)=\left(\begin{array}{c}
  p \\
  j
\end{array}\right)\left(1-\sum\limits_{l=1}^k\mu_l\right)^{p-j}
\left(\sum\limits_{l=1}^k\mu_l\right)^j.
$$
\end{proof}
\begin{rem}It is not hard to check that $\sum\limits_{m=0}^{p}P_{2N}(p,m)=1$,
where $P_{2N}(p,m)$ is given by equation (\ref{n2.2}). Indeed,
it is known that
$
\sum\limits_{m=0}^pP_{2N}^{o}(p,m)=1,
$
so we have
\begin{equation}
\begin{split}
&\sum\limits_{m=0}^pP_{2N}(p,m)
=\left(\begin{array}{c}
  p \\
  p
\end{array}\right)\left(1-\sum\limits_{l=1}^k\mu_l\right)^{p}
\sum\limits_{m=0}^pP_{2N}^{o}(p,m)\\
&+\left(\begin{array}{c}
  p \\
  p-1
\end{array}\right)\left(1-\sum\limits_{l=1}^k\mu_l\right)^{p-1}
\left(\sum\limits_{l=1}^k\mu_l\right)\sum\limits_{m=0}^{p-1}P_{2N}^{o}(p-1,m)\\
&+\ldots+\left(\begin{array}{c}
  p \\
  0
\end{array}\right)\left(1-\sum\limits_{l=1}^k\mu_l\right)^{p-p}
\left(\sum\limits_{l=1}^k\mu_l\right)^p=1.
\end{split}
\end{equation}
\end{rem}
\begin{prop}\label{Proposition5.6.1} Assume that the mutation probabilities $\mu_l$, $l=1,\ldots, k$,
are weak, i.e
\begin{equation}\label{nmu}
\mu_1:=\mu_1(N)=\frac{\theta_1}{4N},\ldots,\mu_k:=\mu_k(N)=\frac{\theta_k}{4N},
\end{equation}
where $\theta_1>0$, $\ldots$, $\theta_k>0$ are fixed numbers. Then
\begin{equation}\label{n22.6}
P_{2N}(p,m)=\left\{
  \begin{array}{ll}
    1-\frac{p(p-1)+\theta_1p+\ldots+\theta_kp}{4N}+O\left(N^{-2}\right), &  m=p,\\
    \frac{p(p-1)+\theta_1p+\ldots+\theta_kp}{4N}+O\left(N^{-2}\right), & m=p-1, \\
    O\left(N^{-2}\right), & \hbox{otherwise,}
  \end{array}
\right.
\end{equation}
as $N\longrightarrow\infty$.
\end{prop}
\begin{proof}The asymptotics of $P_{2N}^{o}(p,m)$ is known, namely
\begin{equation}\label{n2.7}
P_{2N}^{o}(p,m)=\left\{
  \begin{array}{ll}
    1-\frac{p(p-1)}{4N}+O\left(N^{-2}\right), &  m=p,\\
    \frac{p(p-1)}{4N}+O\left(N^{-2}\right), & m=p-1, \\
    O\left(N^{-2}\right), & \hbox{otherwise,}
  \end{array}
\right.
\end{equation}
as $N\longrightarrow\infty$, see, for example, Feng \cite{Feng},
Chapter 4. To obtain (\ref{n22.6}), use (\ref{n2.2}) together
with the asymptotic formulae (\ref{n2.7}) and take (\ref{nmu})
into account.
\end{proof}
A Markov chain $\left\{Z_{n,2N}(t): t=0,1,2,\ldots\right\}$ associated with $P_{2N}(p,m)$ is defined
as follows. Consider the present time as $t=0$. The time $t=j$ will be the time to go back generations $j$ to the past.
Set $Z_{n,2N}(0)=n$, and let $Z_{n,2N}(t)$ be the number of distinct ancestors in generation $t$ of a sample of size $n$ at time $t=0$. The state space of $\left\{Z_{n,2N}(t): t=0,1,2,\ldots\right\}$ is $\{0,1,\ldots,n\}$, the number $0$ is included due to mutations. The probability $P_{2N}(p,m)$
is the transition probability for $Z_{n,2N}(.)$,
\begin{equation}\label{n5.7.3}
\Prob\left(Z_{n,2N}(t+1)=m|Z_{n,2N}(t)=p\right)
=P_{2N}(p,m),\;\; p,m=0,1,\ldots,n.
\end{equation}
\subsection{The ancestral process with mutations}\label{SectionAncestral}
Write
\begin{equation}\label{n5.7.4}
G_{2N}=\left(\begin{array}{cccc}
  P_{2N}(0,0) & P_{2N}(0,1) & \ldots & P_{2N}(0,n) \\
  P_{2N}(1,0) & P_{2N}(1,1) & \ldots & P_{2N}(1,n) \\
  \vdots &  &  &  \\
  P_{2N}(n,0) & P_{2N}(n,1) & \ldots & P_{2N}(n,n)
\end{array}
\right).
\end{equation}
Proposition \ref{Proposition5.6.1} implies
\begin{equation}\label{n5.7.5}
G_{2N}=I+\frac{1}{N}\left(\begin{array}{ccccc}
  0 & 0 & 0 &\ldots & 0 \\
  q_{1,0} & q_{1,1} &  0 & \ldots & 0 \\
  0 & q_{2,1} & q_{2,2} &\ldots & 0\\
  \vdots &  & &  &  \\
  0 & \ldots & 0 & q_{n,n-1} & q_{n,n}
\end{array}
\right)+O\left(\frac{1}{N^2}\right),
\end{equation}
as $N\longrightarrow\infty$. Here
\begin{equation}\label{n5.7.6}
q_{j,j}=-\frac{j(j-1)+\theta_1j+\ldots +\theta_k j}{4},\;\;
q_{j,j-1}=\frac{j(j-1)+\theta_1j+\ldots +\theta_k j}{4},
\end{equation}
where $j\in\left\{1,\ldots,n\right\}$.
Clearly, we have
\begin{equation}\label{n5.7.7}
\underset{N\rightarrow\infty}{\lim}\left(G_{2N}\right)^{Nt}
=e^{Qt},\;\;
Q=\left(\begin{array}{ccccc}
  0 & 0 & 0 &\ldots & 0 \\
  q_{1,0} & q_{1,1} &  0 & \ldots & 0 \\
  0 & q_{2,1} & q_{2,2} &\ldots & 0\\
  \vdots &  & &  &  \\
  0 & \ldots & 0 & q_{n,n-1} & q_{n,n}
\end{array}
\right).
\end{equation}
We conclude that if we rescale the time so that the time is measured
in units of $N$ generations, and let $N\longrightarrow\infty$,
then the Markov chain $Z_{n,2N}(.)$ is approximated by a pure death process $D_n(t)$. It starts from $D_n(0)=n$.
The state space of this process is $\left\{0,1,\ldots,n\right\}$, and the matrix $Q$ defined by
equations (\ref{n5.7.6}), (\ref{n5.7.7}) is the infinitesimal
generator of this process. The process $D_n(t)$ can be called
\textit{the ancestral process with mutations} for a sample of size $n$.

We have
\begin{equation}
\Prob\left(D_n(h)=j-1|D_n(0)=j\right)=
\frac{j(j-1)+\theta_1j+\ldots+\theta_kj}{4}h+o(h),
\end{equation}
as $h\rightarrow 0+$. The expression
$$
\frac{j(j-1)}{4}+\frac{\theta_1 j}{4}+\ldots+\frac{\theta_kj}{4}
$$
is  the transition intensity of the process $D_n(t)$. It is interpreted as
we have a coalescing event with probability
\begin{equation}\label{CoalescentProbability}
\frac{j-1}{j-1+\theta_1+\ldots+\theta_k},
\end{equation}
or a mutation into an allele of class $l$ with probability
\begin{equation}\label{MutationProbability}
\frac{\theta_l}{j-1+\theta_1+\ldots+\theta_k},
\end{equation}
where  $j=1,\ldots,n$, and $l=1,\ldots,k$.

The ancestral process with mutations corresponds
to a generalized Hoppe's urn model described below.


\subsection{A generalization of the Hoppe urn model.}\label{Section3.2}
Consider an urn that contains objects of $k$ different classes
(for example balls, cubes, pyramids, etc.). For each $l\in\{1,\ldots,k\}$
the $l$th class is represented by a black object with mass $\theta_l>0$ and
by a number of colored objects of the same class.
We agree that every colored (non-black) object has mass 1.

Assume that such an urn is given and contains $n$ objects.
A multiple partition $\Lambda_n^{(k)}=\left(\lambda^{(1)},\ldots,\lambda^{(k)}\right)$ of $n$
into $k$ components can be associated with the urn as follows.
Consider colored objects of the first class and denote
by $\lambda_1^{(1)}$ the number of such objects of the most frequent color, by $\lambda_2^{(1)}$ the number of such objects of the next frequent color, etc.
As a result, we will obtain a Young diagram $\lambda^{(1)}$ associated with the objects of the first class in the urn.
The Young diagrams $\lambda^{(2)}$, $\ldots$, $\lambda^{(k)}$
associated with objects of other classes are constructed in the same way.

Let us define a process which is a generalization
of the Hoppe urn process in Ref. \cite{Hoppe1}.
Originally, there are only $k$ black objects in the urn.
Namely, there is one black object of the first class
with mass $\theta_1$, there is one black object of the second class with mass $\theta_2$, $\ldots$,  there is one black object of the $k$th class with mass $\theta_k$.
At each time $j=1,2,\ldots$, an object is selected at random with a probability proportional to its mass. If a colored (non-black) object is selected, that object and another object of the same class and of the same color are returned to the urn. If a black object is chosen, that object is returned to the urn with an object of the same class and of a new color. This new object has mass $1$.

Let $\mathfrak{L}_j^{(k)}$ be a multiple partition associated
with the urn containing objects of $k$ different classes (as described above) at time j. Then $\mathfrak{L}_j^{(k)}$ is a random variable, and
the collection $\left\{\mathfrak{L}_j^{(k)}\right\}_{j=1}^{\infty}$
is a Markov process. The marginal distribution of this process
is given by the next Theorem.
\begin{thm}\label{TheoremGeneralizedHoope}
Let $\lambda^{(1)}$, $\ldots$, $\lambda^{(k)}$ be fixed
Young diagrams such that
$$
\left|\lambda^{(1)}\right|+\ldots+\left|\lambda^{(k)}\right|=n.
$$
Then $\Lambda_n^{(k)}=\left(\lambda^{(1)},\ldots,\lambda^{(k)}\right)$ is a fixed multiple partition of $n$ into $k$ components. We have
\begin{equation}
\Prob\left(\mathfrak{L}_n^{(k)}=\Lambda_n^{(k)}\right)=M_{\theta_1,\ldots,\theta_k;n}^{\Ewens}\left(\Lambda_n^{(k)}\right),
\end{equation}
where $n=1,2,\ldots$, and the probability measure $M_{\theta_1,\ldots,\theta_k;n}^{\Ewens}$ is defined by equations (\ref{MultipleEwensMeasure}) and (\ref{EwensMeasure}).
\end{thm}
\begin{proof}
We proceed by induction. If $n=1$, the possible multiple partitions are
$$
((1),\varnothing,\ldots,\varnothing),\; (\varnothing, (1),\varnothing,\ldots,\varnothing),\ldots,
(\varnothing,\varnothing,\ldots,(1)).
$$
In particular, the multiple partition
$$
\left(\underset{l}{\underbrace{\varnothing,\ldots,\varnothing, (1)}},\varnothing,\ldots,\varnothing\right)
$$
has probability $\frac{\theta_l}{\theta_1+\ldots+\theta_k}$. On the other hand,
\begin{equation}
\begin{split}
&M^{\Ewens}_{\theta_1,\ldots,\theta_k; n=1}\left(\left(\underset{l}{\underbrace{\varnothing,\ldots,\varnothing, (1)}},\varnothing,\ldots,\varnothing\right)
\right)\\
&=\frac{(\theta_l)_1}{(\theta_1+\ldots+\theta_k)_1}
\frac{1!}{0!\ldots 1!\ldots 0!}
M^{\Ewens}_{\theta_1,|\lambda^{(1)}|=0}\left(\varnothing\right)
\ldots
M^{\Ewens}_{\theta_l,|\lambda^{(l)}|=1}\left((1)\right)
\ldots
M^{\Ewens}_{\theta_k,|\lambda^{(k)}|=0}\left(\varnothing\right)\\
&=\frac{\theta_l}{\theta_1+\ldots+\theta_k},
\end{split}
\end{equation}
so the result holds true.

Denote by $P_{n}^{\Hoppe}$ the probability measure describing the distribution of the random multiple
partition $\mathfrak{L}_n^{(k)}$. We have $\Prob\left(\mathfrak{L}_n^{(k)}=\Lambda_n^{(k)}\right)=P_n^{\Hoppe}(\Lambda_n^{(k)})$, and
\begin{equation}\label{P1}
P_n^{\Hoppe}\left(\Lambda_n^{(k)}\right)=\sum\limits_{\Lambda_{n-1}^{(k)}\in\Y_{n-1}^{(k)}}
P^{\Hoppe}_{n-1}\left(\Lambda_{n-1}^{(k)}\right)p\left(\Lambda_{n-1}^{(k)},\Lambda_n^{(k)}\right),
\end{equation}
where
$$
p\left(\Lambda_{n-1}^{(k)},\Lambda_n^{(k)}\right)
=\Prob\left(\mathfrak{L}_n^{(k)}=\Lambda_n^{(k)}|\mathfrak{L}_{n-1}^{(k)}=\Lambda_{n-1}^{(k)}\right)
$$
is the transition probability of the Markov process $\left\{\mathfrak{L}_j^{(k)}\right\}_{j=1}^{\infty}$. Let us find this transition probability explicitly.
We rewrite equation (\ref{P1}) as
\begin{equation}
\begin{split}
&P^{\Hoppe}_n\left(\left(\lambda^{(1)},\ldots,\lambda^{(k)}\right)\right)\\
&=\sum\limits_{\mu^{(1)}\nearrow\lambda^{(1)}}P_{n-1}^{\Hoppe}\left(\left(\mu^{(1)},\lambda^{(2)},\ldots,\lambda^{(k)}\right)\right)
p\left(\left(\mu^{(1)},\lambda^{(2)},\ldots,\lambda^{(k)}\right),\left(\lambda^{(1)},\lambda^{(2)},\ldots,\lambda^{(k)}\right)\right)\\
&+\sum\limits_{\mu^{(2)}\nearrow\lambda^{(2)}}P_{n-1}^{\Hoppe}\left(\left(\lambda^{(1)},\mu^{(2)},\ldots,\lambda^{(k)}\right)\right)
p\left(\left(\lambda^{(1)},\mu^{(2)},\ldots,\lambda^{(k)}\right),\left(\lambda^{(1)},\lambda^{(2)},\ldots,\lambda^{(k)}\right)\right)\\
&+\ldots\\
&+\sum\limits_{\mu^{(k)}\nearrow\lambda^{(k)}}P_{n-1}^{\Hoppe}\left(\left(\lambda^{(1)},\lambda^{(2)},\ldots,\mu^{(k)}\right)\right)
p\left(\left(\lambda^{(1)},\lambda^{(2)},\ldots,\mu^{(k)}\right),\left(\lambda^{(1)},\lambda^{(2)},\ldots,\lambda^{(k)}\right)\right),
\end{split}
\end{equation}
where $\mu^{(l)}\nearrow\lambda^{(l)}$ means that $\lambda^{(l)}$ is obtained from $\mu^{(l)}$ by adding one box.
In order to get a formula for the transition probability
$$
p\left(\left(\lambda^{(1)},\ldots,\mu^{(l)},\ldots,\lambda^{(k)}\right),\left(\lambda^{(1)},\ldots,\lambda^{(l)},\ldots,\lambda^{(k)}\right)\right)
$$
(where $\mu^{(l)}\nearrow\lambda^{(l)}$ ) we observe that $\lambda^{(l)}$ can be obtained from $\mu^{(l)}$ in two ways:\\
(a) A new box is attached to the bottom of $\mu^{(l)}$. If $\lambda^{(l)}
=\left(1^{m_1^{(l)}}2^{m_2^{(l)}}\ldots\right)$, then
$\mu^{(l)}=\left(1^{m_1^{(l)}-1}2^{m_2^{(l)}}\ldots\right)$.
This corresponds to the addition of an object of class $l$ with a new color.
In this case, the transition probability for the generalized Hoppe urn model is
$$
\frac{\theta_l}{\theta_1+\ldots+\theta_k+n-1}.
$$
(b) A row of length $j$ in $\mu^{(l)}$ becomes a row of length $j+1$ in $\lambda^{(l)}$, where $j$ takes values in the set
$\left\{1,2,\ldots,\lambda_1^{(l)}-1\right\}$. If
$$
\lambda^{(l)}=\left(1^{m_1^{(l)}}2^{m_2^{(l)}}...\right),
$$
then
$$
\mu^{(l)}=\left(1^{m_1^{(l)}}2^{m_2^{(l)}}\ldots (j-1)^{m_{j-1}^{(l)}}j^{m_j^{(l)}+1}(j+1)^{m_{j+1}^{(l)}-1}(j+2)^{m_{j+2}^{(l)}}\ldots\right).
$$
In this case, at time $n-1$ a colored (non-black) object of class $l$ is chosen from the existing $j$ objects of class $l$. The chosen object is added with an additional object of the same class, and of the same color. The corresponding transition probability is
$$
\frac{j\left(m_j^{(l)}+1\right)}{\theta_1+\ldots+\theta_k+n-1}.
$$
It follows that the distribution $P^{\Hoppe}_n$ of the generalized Hoppe urn satisfies the recursion
\begin{equation}
\begin{split}
&1=\sum\limits_{l=1}^k\frac{\theta_l}{\theta_1+\ldots+\theta_k+n-1}
\frac{P^{\Hoppe}_{n-1}\left(1^{m_1^{(1)}}2^{m_2^{(1)}}\ldots;1^{m_1^{(l)}-1}2^{m_2^{(l)}}\ldots;
1^{m_1^{(k)}}2^{m_2^{(k)}}\ldots\right)}{P^{\Hoppe}_{n}\left(1^{m_1^{(1)}}2^{m_2^{(1)}}\ldots;1^{m_1^{(l)}}2^{m_2^{(l)}}\ldots;
1^{m_1^{(k)}}2^{m_2^{(k)}}\ldots\right)}\\
&+\sum\limits_{l=1}^k\sum\limits_{j=1}^{\lambda_1^{(l)}-1}
\frac{j(m_j^{(l)}+1)P^{\Hoppe}_{n-1}\left(1^{m_1^{(1)}}2^{m_2^{(1)}}\ldots;\ldots;1^{m_1^{(l)}}\ldots j^{m_j^{(l)}+1}(j+1)^{m_{j+1}^{(l)}-1}\ldots;\ldots;
1^{m_1^{(k)}}2^{m_2^{(k)}}\ldots\right)}{(\theta_1+\ldots+\theta_k+n-1)P^{\Hoppe}_{n}\left(1^{m_1^{(1)}}2^{m_2^{(1)}}\ldots;\ldots;1^{m_1^{(l)}}2^{m_2^{(l)}}\ldots;
\ldots;1^{m_1^{(k)}}2^{m_2^{(k)}}\ldots\right)}.
\end{split}
\nonumber
\end{equation}
Let us consider the probability measure $M^{\Ewens}_{\theta_1,\ldots,\theta_k;n}$ defined by equations (\ref{MultipleEwensMeasure}) and (\ref{EwensMeasure}). It is enough to show that
$M^{\Ewens}_{\theta_1,\ldots,\theta_k;n}$ satisfies the same recursion as $P^{\Hoppe}_{n}$.
We have
\begin{equation}
\frac{M^{\Ewens}_{\theta_1,\ldots,\theta_k;n-1}\left(1^{m_1^{(1)}}2^{m_2^{(1)}}\ldots;1^{m_1^{(l)}-1}2^{m_2^{(l)}}\ldots;
1^{m_1^{(k)}}2^{m_2^{(k)}}\ldots\right)}{M^{\Ewens}_{\theta_1,\ldots,\theta_k;n}\left(1^{m_1^{(1)}}2^{m_2^{(1)}}\ldots;1^{m_1^{(l)}}2^{m_2^{(l)}}\ldots;
1^{m_1^{(k)}}2^{m_2^{(k)}}\ldots\right)}=\frac{m_1^{(l)}\left(\theta_1+\ldots+\theta_k+n-1\right)}{n\theta_l},
\end{equation}
and
\begin{equation}
\begin{split}
&\frac{M^{\Ewens}_{\theta_1,\ldots,\theta_k;n-1}\left(1^{m_1^{(1)}}2^{m_2^{(1)}}\ldots;\ldots;1^{m_1^{(l)}}\ldots j^{m_j^{(l)}+1}(j+1)^{m_{j+1}^{(l)}-1}\ldots;\ldots;
1^{m_1^{(k)}}2^{m_2^{(k)}}\ldots\right)}{M^{\Ewens}_{\theta_1,\ldots,\theta_k;n}\left(1^{m_1^{(1)}}2^{m_2^{(1)}}\ldots;\ldots;1^{m_1^{(l)}}2^{m_2^{(l)}}\ldots;
\ldots;1^{m_1^{(k)}}2^{m_2^{(k)}}\ldots\right)}\\
&=\frac{\theta_1+\ldots+\theta_k+n-1}{n}\;\frac{j+1}{j}
\frac{m_{j+1}^{(l)}}{m_j^{(l)}+1}.
\end{split}
\end{equation}
We see that $M^{\Ewens}_{\theta_1,\ldots,\theta_k;n}$ satisfies the same recurrence relation as $P^{\Hoppe}_{n}$ provided that the condition
\begin{equation}\label{condition1}
1=\sum\limits_{l=1}^{k}\frac{m_1^{(l)}}{n}+\sum\limits_{l=1}^k\sum\limits_{j=1}^{\lambda_1^{(l)}-1}\frac{(j+1)m_{j+1}^{(l)}}{n}.
\end{equation}
is satisfied. Since
$$
m_1^{(l)}+\sum\limits_{j=1}^{\lambda_1^{(l)}-1}(j+1)m_{j+1}^{(l)}=\sum\limits_{j=1}^{\lambda_1^{(l)}}jm_j^{(l)}=|\lambda^{(l)}|,
$$
and $|\lambda^{(1)}|+\ldots+|\lambda^{(k)}|=n$, condition (\ref{condition1}) holds true indeed.
\end{proof}
Finally, interpret $\mathfrak{L}_n^{(k)}$ from Theorem \ref{TheoremGeneralizedHoope} as a random matrix describing the allelic
composition of a sample of size $n$ (see the beginning of Section \ref{SectionRESFMPS})
whose genealogy is described by coalescence with probability (\ref{CoalescentProbability}), and mutations with
probabilities (\ref{MutationProbability}).  With such an interpretation, Theorem \ref{TheoremGeneralizedHoope} and Proposition \ref{PropositionIndeedDistribution} say that the joint probability distribution of the random variables
$A_j^{(l)}(n)$
is given by the refined Ewens sampling formula, equation (\ref{RefinedESF}). Theorem \ref{MAINTHEOREM} is proved.
\qed

\section{The Ewens multiple partition structure and the multiple Poisson-Dirichlet distribution}\label{SectionDerivation1}

In this Section we present a second derivation of formula (\ref{RefinedESF}) in the context of a Wright-Fisher-type model with infinite many allelic types
and mutations.

Our second derivation
works forward in time: we compute the probability of the desired event by conditioning on the allele composition at the sampling time and then we integrate
this probability according to the equilibrium measure for the allele frequencies.
We show that this equilibrium measure should be the multiple Poisson-Dirichlet distribution $PD(\theta_1,\ldots,\theta_k)$.
Our derivation here uses heuristic arguments. However,
it provides a representation of the multiple Ewens distribution in terms of the multiple Poisson-Dirichlet distribution and
an intuitive explanation of this representation.

Consider an infinite population divided into
$k$ classes, where each class is a union of a countable number of sub-populations characterized by their own allelic types. Thus, a subpopulation is labeled by
the pair $(j,l)$ of numbers, indicating that it is characterized by the $j$th allelic type of the mutation class numbered $l$. Here $j=1,2,\ldots,$ and $l=1,\ldots,k$.

Suppose that the proportion of the population labeled by $(j,l)$ is $x_j^{(l)}$. Then
$x_j^{(l)}\geq 0$,  and
\begin{equation}\label{111.2.1}
\begin{split}
&x_1^{(1)}+x_2^{(1)}+\ldots=\delta^{(1)}\\
\vdots\\
&x_1^{(k)}+x_2^{(k)}+\ldots=\delta^{(k)}.
\end{split}
\end{equation}
The parameter $\delta^{(l)}$ characterizes the frequency of the $l$th class in the population, so the condition
\begin{equation}\label{111.2.2}
\delta^{(1)}+\ldots+\delta^{(k)}=1
\end{equation}
should be satisfied.

Let $n$ be the size of a sample chosen from the population. The alleles presented in the sample can be characterized by a multiple partition
\begin{equation}\label{LYM1}
\Lambda_n^{(k)}=\left(\lambda^{(1)},\ldots,\lambda^{(k)}\right),
\;\;|\lambda^{(1)}|+\ldots+|\lambda^{(k)}|=n.
\end{equation}
Recall that the meaning of $\Lambda_n^{(k)}$ is the following. Write
\begin{equation}\label{LYM2}
\lambda^{(l)}=\left(1^{m_1^{(l)}}2^{m_2^{(l)}}\ldots n^{m_n^{(l)}}\right),
\end{equation}
for each $l\in\left\{1,\ldots,k\right\}$. Then $m_j^{(l)}$ is the number of alleles of the $l$th class that appear $j$ times in the sample.

If a sample of size $n$ is taken at random from the population, the corresponding multiple partition that describes the allelic composition of the sample is a random variable
$\mathfrak{L}_n^{(k)}$ taking values in $\Y_n^{(k)}$. Here we show that the distribution of $\mathfrak{L}_n^{(k)}$ is given by $M^{\Ewens}_{\theta_1,\ldots,\theta_k; n}$.
This is equivalent to the refined Ewens sampling formula (equation (\ref{RefinedESF}))
as explained in Section \ref{SectionProof}.

Assume first that the frequencies $x_j^{(l)}$ introduced above are fixed positive numbers satisfying (\ref{111.2.1}). The probability that for each $l=1,\ldots,k$  a random sample of size $n$ has $m_1^{(l)}$ of alleles of the class
$l$ represented once, $m_2^{(l)}$ alleles of the class $l$ represented twice and so on is
\begin{equation}
\begin{split}
&\mathbb{K}\left(\Lambda_n^{(k)},x^{(1)},\ldots,x^{(k)}\right)
=\frac{n!}{\prod_{l=1}^k\prod_{j=1}^n\left(j!\right)^{m_j^{(l)}}}\\
&\times
\sum\left[\left(x_1^{(1)}\right)^{\nu_1^{(1)}}\left(x_2^{(1)}\right)^{\nu_2^{(1)}}\ldots\right]
\ldots\left[\left(x_1^{(k)}\right)^{\nu_1^{(k)}}\left(x_2^{(k)}\right)^{\nu_2^{(k)}}\ldots\right],
\end{split}
\end{equation}
where  $\Lambda_n^{(k)}$ is defined by (\ref{LYM1}) and (\ref{LYM2}), the sum is over all $k$ sequences $\left(\nu_1^{(1)},\nu_2^{(1)},\ldots\right)$,
$\ldots$, $\left(\nu_1^{(k)},\nu_2^{(k)},\ldots\right)$ of integers $0\leq\nu_j^{(l)}\leq n$
such that the number of elements equal to $j$ in $\left(\nu_1^{(l)},\nu_2^{(l)},\ldots\right)$
is precisely $m_j^{(l)}$, where $l=1,\ldots,k$, and $j=1,\ldots,n$. We can write
\begin{equation}
\begin{split}
&\mathbb{K}\left(\Lambda_n^{(k)},x^{(1)},\ldots,x^{(k)}\right)
=\frac{n!}{\prod_{l=1}^k\prod_{j=1}^n\left(j!\right)^{m_j^{(l)}}}\;
m_{\lambda^{(1)}}\left(x^{(1)}\right)\ldots m_{\lambda^{(k)}}\left(x^{(k)}\right),
\end{split}
\end{equation}
where $m_{\lambda^{(l)}}\left(x^{(l)}\right)$ is the monomial symmetric function.

Next, we allow $x^{(1)}$, $\ldots$, $x^{(k)}$ to be random. If we restrict ourselves to samples containing alleles from class $l$, then the corresponding frequencies should be governed by the Poisson-Dirichlet  distribution\footnote{For the definition of the Poisson-Dirichlet distribution, and its connection
to the Ewens sampling formula we refer the reader to book by Feng \cite{Feng}.} $PD(\theta_l)$, where $\theta_l$ is the parameter characterizing mutation in class $l$. The parameters $\delta_1$, $\ldots$, $\delta_k$ are random variables and are interpreted as proportions of sub-populations.
Clearly, these random variables should satisfy
\begin{equation}
\mathbb{E}\left(\delta_1\right)=\frac{\theta_1}{\theta_1+\ldots+\theta_k},\ldots,\mathbb{E}\left(\delta_k\right)=\frac{\theta_k}{\theta_1+\ldots+\theta_k}.
\end{equation}
The Dirichlet distribution $D\left(\theta_1,\ldots,\theta_k\right)$ is a suitable choice for the joint distribution of the random variables $\delta_1$, $\ldots$, $\delta_k$ (see, for example,
the book by Kotz, Balakrishnan and Johnson \cite{KotzBalakrishnanJohnson}, Chapter 49,  where different properties and applications of the Dirichlet distribution are described. In particular,   see the discussion in Chapter 49, $\S$5 of this book on characterizations of the Dirichlet distribution).

This suggests the following description of random frequencies
$x^{(l)}=\left(x_1^{(l)},x_2^{(l)},\ldots\right)$, $l=1,\ldots,k$. Let us assume that for each $l=1,\ldots,k$ the sequence
$y^{(l)}=\left(y_1^{(l)}\geq y_2^{(l)}\ldots\geq 0\right)$ of random variables is governed by the Poisson-Dirichlet distribution by $PD\left(\theta_l\right)$, and the sequences $y^{(1)}$, $\ldots$, $y^{(k)}$ are independent. In addition, let $\delta_1$, $\ldots$, $\delta_k$ be random variables independent of $y^{(1)}$, $\ldots$, $y^{(k)}$ whose joint distribution is $D(\theta_1,\ldots,\theta_k)$. Then each sequence $x^{(l)}$ is distributed as
$$
\delta_ly^{(l)}=\left(\delta_ly_1^{(l)}\geq\delta_ly_2^{(l)}\geq \ldots\geq 0\right).
$$
The joint distribution of $x^{(1)}$, $\ldots$, $x^{(k)}$ is called the \textit{multiple Poisson-Dirichlet distribution}, and it is denoted by $PD\left(\theta_1,\ldots,\theta_k\right)$.
\begin{prop}\label{PropositionMEwenSASIntegral}
We have
\begin{equation}\label{R1.7}
\int\limits_{\overline{\nabla}_0^{(k)}}\mathbb{K}\left(\Lambda_n^{(k)},\omega\right)PD(\theta_1,\ldots,\theta_k)(d\omega)=M_{\theta_1,\ldots,\theta_k;n}^{\Ewens}\left(\Lambda_n^{(k)}\right),
\end{equation}
where $\overline{\nabla}_0^{(k)}$ is defined by
\begin{equation}\label{SetNabla}
\begin{split}
\overline{\nabla}^{(k)}_0=&\biggl\{(x,\delta)\biggr|
x=\left(x^{(1)},\ldots,x^{(k)}\right),\;
\delta=\left(\delta^{(1)},\ldots,\delta^{(k)}\right);\\
&x^{(l)}=\left(x^{(l)}_1\geq x^{(l)}_2\geq\ldots\geq 0\right),\;\delta^{(l)}\geq 0,\;
1\leq l\leq k,\\
&\mbox{where}\;\;\sum\limits_{i=1}^{\infty}x_i^{(l)}=\delta^{(l)},\; 1\leq l\leq k,\;\mbox{and}\;\sum\limits_{l=1}^k\delta^{(l)}=1\biggl\}.
\end{split}
\end{equation}
\end{prop}
\begin{proof}The integration variables in the left-hand side of equation (\ref{R1.7}) can be represented as $x^{(l)}=\delta_ly^{(l)}$, $l=1,\ldots,k$.  The joint distribution of $\delta_1$, $\ldots$, $\delta_k$ is $D\left(\theta_1,\ldots,\theta_k\right)$,
and the distribution of $y^{(l)}$ is $PD(\theta_l)$. The left-hand side of equation (\ref{R1.7}) can be written as
\begin{equation}
\begin{split}
\frac{n!}{\prod_{l=1}^k\prod_j^n(j!)^{m_j^{(l)}}}
\int\limits_{\overline{\nabla}_0^{(1)}}\ldots
\int\limits_{\overline{\nabla}_0^{(1)}}
&\underset{\delta_1+\ldots+\delta_k=1}{\underset{\delta_1\geq 0,\ldots,\delta_k\geq 0}{\int\ldots\int}}
\left(\prod\limits_{l=1}^km_{\lambda^{(l)}}\left(\delta_ly^{(l)}\right)PD\left(\theta_l\right)(dy^{(l)})\right)\\
&\times D(\theta_1,\ldots,\theta_k)\left(d\delta_1\ldots d\delta_k\right).
\end{split}
\end{equation}
Note that
\begin{equation}
m_{\lambda^{(l)}}\left(\delta_ly^{(l)}\right)=\left(\delta_l\right)^{|\lambda^{(l)}|}m_{\lambda^{(l)}}\left(y^{(l)}\right),
\end{equation}
and that
\begin{equation}
\begin{split}
&\underset{\delta_1+\ldots+\delta_k=1}{\underset{\delta_1\geq 0,\ldots,\delta_k\geq 0}{\int\ldots\int}}
\delta_1^{|\lambda^{(1)}|}\ldots\delta_k^{|\lambda^{(k)}|}D(\theta_1,\ldots,\theta_k)\left(d\delta_1\ldots d\delta_k\right)=\frac{(\theta_1)_{|\lambda^{(1)}|}\ldots(\theta_k)_{|\lambda^{(k)}|}}{(\theta_1+\ldots+\theta_k)_n}.
\end{split}
\end{equation}
Also, the integral representation of the Ewens distribution $M_{\theta_l,n}^{\Ewens}$ is
\begin{equation}
\frac{|\lambda^{(l)}|!}{\prod_{j=1}^n(j!)^{m_j^{(l)}}}
\int\limits_{\overline{\nabla}_0^{(1)}}m_{\lambda^{(l)}}\left(y^{(l)}\right)PD(\theta_l)(dy^{(l)})=M_{\theta_l,|\lambda^{(l)}|}^{\Ewens}\left(\lambda^{(l)}\right),
\end{equation}
see Kingman \cite{Kingman1}.
Therefore, the left-hand side of equation (\ref{R1.7}) is equal to
$$
\frac{(\theta_1)_{|\lambda^{(1)}|}\ldots(\theta_k)_{|\lambda^{(k)}|}}{(\theta_1+\ldots+\theta_k)_n}
\frac{n!}{|\lambda^{(1)}|!\ldots|\lambda^{(k)}|!}
M_{\theta_1,|\lambda^{(1)}|}^{\Ewens}\left(\lambda^{(1)}\right)\ldots M_{\theta_k,|\lambda^{(k)}|}^{\Ewens}\left(\lambda^{(k)}\right),
$$
which is $M_{\theta_1,\ldots,\theta_k;n}^{\Ewens}\left(\Lambda_n^{(k)}\right)$, where $\Lambda_n^{(k)}=\left(\lambda^{(1)},\ldots,\lambda^{(k)}\right)$.
\end{proof}
Note that equation (\ref{R1.7}) provides an integral representation for the Ewens multiple partition structure
$\left(M^{\Ewens}_{\theta_1,\ldots,\theta_k;n}\right)_{n=1}^{\infty}$.
\begin{rem} The reader might note that the multiple Poisson-Dirichlet distribution
$PD(\theta_1,\ldots,\theta_k)$ is derived above by heuristic arguments. However, our proof in Section \ref{SectionProof}, equation (\ref{R1.7}), and Theorem 2.2 in Strahov \cite{StrahovMPS} establish rigorously that the equilibrium
measure for the frequencies $x_j^{(l)}$
is $PD\left(\theta_1,\ldots,\theta_k\right)$. Indeed, $\left(M_{\theta_1,\ldots,\theta_k;n}^{\Ewens}\right)_{n=1}^{\infty}$ is a multiple partition structure, and Theorem 2.2  in Strahov \cite{StrahovMPS} says that there is a unique probability
measure on $\overline{\nabla}_0^{(k)}$ for which equation (\ref{R1.7}) is satisfied.
\end{rem}
\section{Consistency of $M_{\theta_1,\ldots,\theta_k;n}^{\Ewens}$: the reduction to the  Ewens distribution on Young diagrams}\label{SectionConsistency}
The interpretation of $\theta_1$, $\ldots$, $\theta_k$ as parameters describing mutations in alleles of different classes dictates a consistency condition on
$M_{\theta_1,\ldots,\theta_k;n}^{\Ewens}$.  Assume that $M_{\theta_1,\ldots,\theta_k;n}^{\Ewens}$ defines the distribution of a random multiple partition describing the allelic composition of a sample of size $n$.
If experimental techniques do not enable distinguishing between $k$ different classes of alleles, then the allelic composition of a sample
should be described by a Young diagram with $n$ boxes distributed according to the usual Ewens formula with the parameter $\theta_1+\ldots+\theta_k$.
More precisely, let $\lambda^{(1)}$, $\ldots$, $\lambda^{(k)}$ be Young diagrams.
Write $\lambda^{(1)}$, $\ldots$, $\lambda^{(k)}$ as
\begin{equation}
\begin{split}
&\lambda^{(1)}=1^{m_1^{(1)}}2^{m_2^{(1)}}\ldots n^{m_n^{(1)}},\\
&\vdots\\
&\lambda^{(k)}=1^{m_1^{(k)}}2^{m_2^{(k)}}\ldots n^{m_n^{(k)}},
\end{split}
\end{equation}
where $n=|\lambda^{(1)}|+\ldots+|\lambda^{(k)}|$. The union
$\lambda^{(1)}\cup\ldots\cup\lambda^{(k)}$ of $\lambda^{(1)}$, $\ldots$, $\lambda^{(k)}$ is defined by
$$
\lambda^{(1)}\cup\ldots\cup\lambda^{(k)}=
1^{m_1^{(1)}+\ldots+m_1^{(k)}}2^{m_2^{(1)}+\ldots+m_2^{(k)}}\ldots n^{m_n^{(1)}+\ldots+m_n^{(k)}}.
$$
In words, the Young diagram $\lambda^{(1)}\cup\ldots\cup\lambda^{(k)}$ is obtained from $\lambda^{(1)}$, $\ldots$, $\lambda^{(k)}$ by combining together all rows of $\lambda^{(1)}$, $\ldots$, $\lambda^{(k)}$.
With this notation,
the following proposition should be valid.
\begin{prop}
Let $M_{\theta_1,\ldots,\theta_k;n}^{\Ewens}$ be the probability
measure on $\Y_n^{(k)}$ defined by equation (\ref{MultipleEwensMeasure}).
Then we have
\begin{equation}\label{7.2.1}
\underset{\lambda^{(1)}\cup\ldots\cup\lambda^{(k)}=\lambda}{\sum\limits_{\Lambda_n^{(k)}
=\left(\lambda^{(1)},\ldots,\lambda^{(k)}\right)}}
M_{\theta_1,\ldots,\theta_k;n}^{\Ewens}\left(\Lambda_n^{(k)}\right)=
M_{\theta_1+\ldots+\theta_k;n}^{\Ewens}\left(\lambda\right),
\end{equation}
where $M_{\theta;n}^{\Ewens}\left(\lambda\right)$ is defined by equation (\ref{EwensMeasure}).
\end{prop}
\begin{proof} Writing the left-hand side of equation (\ref{7.2.1})
explicitly we obtain
\begin{equation}
\begin{split}
&\frac{n!}{\left(\theta_1+\ldots+\theta_k\right)_n}
\underset{m_n^{(1)}+\ldots+m_n^{(k)}=m_n}{\underset{\vdots}{\sum\limits_{m_1^{(1)}+\ldots+m_1^{(k)}=m_1}}}
\frac{\theta_1^{m_1^{(1)}+\ldots+m_n^{(1)}}\ldots
\theta_k^{m_1^{(k)}+\ldots+m_n^{(k)}}}{\left(\prod\limits_{j=1}^n
j^{m_j^{(1)}}m_j^{(1)}!\right)\ldots
\left(\prod\limits_{j=1}^n
j^{m_j^{(k)}}m_j^{(k)}!\right)}\\
&=\frac{n!}{\left(\theta_1+\ldots+\theta_k\right)_n\prod\limits_{j=1}^nj^{m_j}}
\left(\sum\limits_{m_1^{(1)}+\ldots+m_1^{(k)}=m_1}
\frac{\theta_1^{m_1^{(1)}}\ldots\theta_k^{m_1^{(k)}} }{m_1^{(1)}!\ldots m_1^{(k)}!}\right)
\ldots
\left(\sum\limits_{m_n^{(1)}+\ldots+m_n^{(k)}=m_n}
\frac{\theta_1^{m_n^{(1)}}\ldots\theta_k^{m_n^{(k)}} }{m_n^{(1)}!\ldots m_n^{(k)}!}\right)\\
&=\frac{n!}{\left(\theta_1+\ldots+\theta_k\right)_n}
\frac{\left(\theta_1+\ldots+\theta_k\right)^{m_1+\ldots+m_n}}{\prod\limits_{j=1}^nj^{m_j}m_j!}.
\end{split}
\nonumber
\end{equation}
The last expression is $M_{\theta_1+\ldots+\theta_k; n}^{\Ewens}\left(\lambda\right)$, see equation (\ref{EwensMeasure}).
\end{proof}

\section{The refined Ewens sampling formula as a distribution on a wreath product. The Chinese restaurant process for wreath products}\label{SectionChineseMain}
\subsection{The refined Ewens sampling formula as a distribution on $G\sim S(n)$}\label{SectionChinese}
The Ewens sampling formula (equation (\ref{ESF})) is closely related to the Ewens measure $P_{\theta,n}^{\Ewens}$ on the symmetric group $S(n)$. This measure is defined by
$$
P_{\theta,n}^{\Ewens}(s)=\frac{\theta^{l(s)}}{\theta(\theta+1)\ldots (\theta+n-1)},\;\; s\in S(n),
$$
where $\theta>0$ is a parameter and $l(s)$ is the number of cycles in $s$. The link between $P_{\theta,n}^{\Ewens}$ and formula (\ref{ESF}) is simple:
to a permutation $s\in S(n)$ assign a Young diagram $\lambda$ with $n$ boxes whose rows represent the cycle-lengths of $s$. The projection $s\rightarrow\lambda$ takes the probability measure $P_{\theta,n}^{\Ewens}$ to a probability measure $M_{\theta;n}^{\Ewens}$ on $\Y_n$.  If $\lambda=\left(1^{a_1}2^{a_2}\ldots n^{a_n}\right)$, then
$M_{\theta,n}^{\Ewens}$ is defined by the right-hand side of equation (\ref{ESF}).

In this Section we relate the refined Ewens sampling formula (equation (\ref{RefinedESF})) with a probability measure
on the wreath product $G\sim S(n)$.
Let $G$ be a finite group. Denote by $G_{\ast}$ the set of conjugacy classes in $G$, and
 assume that $G_{\ast}$ consists of $k$ conjugacy classes, namely $G_{\ast}=\left\{c_1,\ldots,c_k\right\}$.
Let $S(n)$ be the symmetric group of degree $n$, i.e.
the permutation group of $\left\{1,2,\ldots,n\right\}$.
The \textit{wreath product} $G\sim S(n)$ is a group whose underlying set is
$$
G^n\times S(n)=\left\{\left(\left(g_1,\ldots,g_n\right),s\right):\;g_i\in G, s\in S(n)\right\}.
$$
The multiplication in $G\sim S(n)$ is defined by
$$
\left((g_1,\ldots,g_n),s\right)\left((h_1,\ldots,h_n),t\right)=
\left((g_1h_{s^{-1}(1)},\ldots,g_nh_{s^{-1}(n)}),st\right).
$$
When $n=1$, $G\sim S(1)$ is $G$. The number of elements in $G\sim S(n)$ is equal to $|G|^nn!$.

Let
\begin{equation}\label{Rest1}
x=\left(\left(g_1,\ldots,g_n\right),s\right)\in G\sim S(n).
\end{equation}
The permutation $s$ can be written as a product of disjoint cycles. If
$(i_1,i_2,\ldots,i_r)$ is one of these cycles, then the element
$g_{i_{r}}g_{i_{r-1}}\ldots g_{i_1}$ is called \textit{the cycle-product of $(i_1,i_2,\ldots,i_r)$}.
If $c_l$ is a conjugacy class in $G$, then
a cycle $(i_1,\ldots,i_r)$  of $s$ is called \textit{of type $c_l$}
if its cycle-product   belongs to $c_l$.  We denote by $[x]_{c_l}$
the number of cycles of $s$ whose cycle-products belong to $c_l$.

Both the conjugacy classes and the irreducible representations of $G\sim S(n)$ are parameterized by multiple partitions $\Lambda_n^{(k)}=\left(\lambda^{(1)},\ldots,\lambda^{(k)}\right)$,
where $k$ is the number of conjugacy classes in $G$, and where $|\lambda^{(1)}|+\ldots+|\lambda^{(k)}|=n$, see, for example, Macdonald
\cite{Macdonald}, Appendix B. In particular, if
$x=\left(\left(g_1,\ldots,g_n\right),s\right)$ belongs to the conjugacy class
 of $G\sim S(n)$  parameterized by $\Lambda_n^{(k)}=\left(\lambda^{(1)},\ldots,\lambda^{(k)}\right)$, then
$$
\lambda^{(l)}=\left(1^{m_1^{(l)}(s)}2^{m_2^{(l)}(s)}\ldots n^{m_n^{(l)}(s)}\right),
$$
where $m_j^{(l)}(s)$ is equal to the number of cycles of length $j$, and of type $c_l$ in $s$.

To an element $x$ of $G\sim S(n)$ assign a multiple partition
$\Lambda_n^{(k)}$ that labels the conjugacy class of $x$. The projection
$$
x\longrightarrow\Lambda_n^{(k)}=\left(\lambda^{(1)},\ldots,\lambda^{(k)}\right)
$$
takes any central probability measure $P_n^{(k)}$ on $G\sim S(n)$
(central measures are those invariant under the conjugation action of $G\sim S(n)$) to a probability measure $M_n^{(k)}$
on the set  $\Y_n^{(k)}$ of all multiple partitions of $n$ into $k$ components.
This establishes a one-to-one correspondence
\begin{equation}\label{CorrespondenceMeasures}
P_n^{(k)}\longleftrightarrow M_n^{(k)}
\end{equation}
between arbitrary central probability measures on $G\sim S(n)$, and arbitrary probability measures on $\Y_n^{(k)}$.
\begin{prop}\label{PropositionCorMeasures}
Under correspondence (\ref{CorrespondenceMeasures}), the probability measure $M^{\Ewens}_{\theta_1,\ldots,\theta_k}$
defined by equation (\ref{MEM})
turns into a probability measure
\begin{equation}\label{PEwens}
P_{t_1,\ldots,t_k;n}^{\Ewens}(x)=\frac{t_1^{[x]_{c_1}}t_2^{[x]_{c_2}}\ldots t_k^{[x]_{c_k}}}{|G|^n\left(\frac{t_1}{\zeta_{c_1}}+\ldots+\frac{t_k}{\zeta_{c_k}}\right)_n}
\end{equation}
on $G\sim S(n)$. Here
$$
\zeta_{c_l}=\frac{|G|}{|c_l|},\;\; t_l=\frac{|G|}{|c_l|}\theta_l,\;\; l=1,\ldots,k.
$$
\end{prop}
\begin{proof} See Strahov \cite{StrahovMPS},
Section 4.
\end{proof}
\begin{rem} The probability measure $P_{t_1,\ldots,t_k;n}^{\Ewens}$ was used in Strahov
\cite{StrahovGRR} to construct generalized regular representations of big wreath products
$G\sim S(\infty)$, and to develop a harmonic analysis on $G\sim S(\infty)$.
\end{rem}
\subsection{An analogue of the Chinese restaurant process}
\label{SectionChinese2}

Our aim here is to construct a process that will generate elements of $G\sim S(n)$ with the distribution defined by equation (\ref{PEwens}).
This process will be an analogue of the Chinese
restaurant process \cite{Aldous}, which is well known in the context of
the symmetric group.

On the first step of the process one
obtains an element $x=(g,(1))\in G\sim S(1)$ with probability
$$
P^{\Ewens}_{t_1,\ldots,t_k; n=1}(x)
=
\left\{
  \begin{array}{ll}
    \frac{t_1}{|c_1|t_1+\ldots+|c_k|t_k}, & g\in c_1, \\
    \vdots &  \\
    \frac{t_k}{|c_1|t_1+\ldots+|c_k|t_k}, & g\in c_k.
  \end{array}
\right.
$$
Next, let us describe how an element $x'\in G\sim S(n+1)$ is obtained from a given element $x\in G\sim S(n)$
represented  as in equation (\ref{Rest1}).
Write $s\in S(n)$ as a union of disjoint cycles,
\begin{equation}\label{8.2.2}
s=\left(i_1i_2\ldots i_{m_1}\right)\left(j_1j_2\ldots j_{m_2}\right)
\ldots\left(\nu_1\nu_2\ldots \nu_{m_p}\right),
\end{equation}
and represent the elements $(g_1, \ldots, g_n)$ in (\ref{Rest1}) as a collection of ordered lists (where each list corresponds to a relevant cycle of $s$)
\begin{equation}\label{8.2.3}
\left(g_{i_1},g_{i_2},\ldots, g_{i_{m_1}}\right)\left(g_{j_1},g_{j_2},\ldots, g_{j_{m_2}}\right)
\ldots\left(g_{\nu_1},g_{\nu_2},\ldots, g_{\nu_{m_p}}\right).
\end{equation}
Clearly, $x\in G\sim S(n)$ represented as in (\ref{Rest1}) determines
(\ref{8.2.2}), (\ref{8.2.3}) uniquely. Conversely,
collection (\ref{8.2.2}) of cycles, and collection (\ref{8.2.3})
of ordered lists of elements of $G$ gives us unambiguously
an element $x$ of $G\sim S(n)$.

The process is such that an element $x\in G\sim S(n)$ can be lifted to an element
of $G\sim S(n+1)$ in two ways. First, we can insert $n+1$ to an existing cycle of $s$,  and  change the corresponding lists
of elements of $G$ accordingly. If
$$
\left(i_1i_2\ldots i_si_{s+1}\ldots i_{m_1}\right)
\rightarrow\left(i_1i_2\ldots i_s\; n+1\; i_{s+1}\ldots i_{m_1}\right),
$$
then we replace
\begin{equation}\label{Element1}
\left(g_{i_1},g_{i_2},\ldots,g_{i_s},g_{i_{s+1}},\ldots g_{i_{m_1}}\right)
\end{equation}
by
\begin{equation}\label{Element2}
\left(g_{i_1},g_{i_2},\ldots,g_{i_s}\left(g_{n+1}\right)^{-1}, g_{n+1},g_{i_{s+1}},\ldots g_{i_{m_1}}\right).
\end{equation}
In this way a new element $x'\in G\sim S(n+1)$ will be obtained.
Note that the condition
 $[x']_{c_l}=[x]_{c_l}$ is satisfied for $l=1,\ldots,k$ in this case. The probability that $x'\in G\sim S(n+1)$ is obtained
 from a given $x\in G\sim S(n)$ as described above is assigned to be
 equal to
$$
\frac{1}{|c_1|t_1+\ldots+|c_k|t_k+n|G|}.
$$
Second, the integer $n+1$ can start a new cycle. In this case
the permutation $s$ is replaced by
\begin{equation}\label{sprime}
s'=\left(i_1i_2\ldots i_{m_1}\right)\left(j_1j_2\ldots j_{m_2}\right)
\ldots\left(\nu_1\nu_2\ldots \nu_{m_p}\right)(n+1)\in S(n+1).
\end{equation}
In words, $s'\in S(n+1)$ is obtained from $s\in S(n)$ by attaching the new cycle $(n+1)$. The collection of ordered lists (\ref{8.2.3}) is replaced by
\begin{equation}\label{newcollection}
\left(g_{i_1},g_{i_2},\ldots, g_{i_{m_1}}\right)\left(g_{j_1},g_{j_2},\ldots, g_{j_{m_2}}\right)
\ldots\left(g_{\nu_1},g_{\nu_2},\ldots, g_{\nu_{m_p}}\right)(g_{n+1}).
\end{equation}
Expressions (\ref{sprime}) and (\ref{newcollection}) determine an element $x'\in G\sim S(n+1)$.
If $g_{n+1}\in c_l$, then $[x']_{c_l}=[x]_{c_l}+1$, and the probability of the event that $x'$ is obtained from $x$ is assigned to be equal to
$$
\frac{t_l}{|c_1|t_1+\ldots+|c_k|t_k+n|G|}.
$$
If $G$ contains the unit element only, then the wreath product $G\sim S(n)$ is isomorphic to the symmetric group $S(n)$, and the process described above coincides with the Chinese restaurant process.  For this reason we will refer to the process described above as to the \textit{generalized
Chinese restaurant process}.
\begin{prop}
The generalized
Chinese restaurant process generates the distribution on $G\sim S(n)$ given by formula (\ref{PEwens}).
\end{prop}
\begin{proof} The probability measure
$P_{t_1,\ldots,t_k;n}^{\Ewens}$
on $G\sim S(n)$ defined by equation (\ref{PEwens}) can be represented as
\begin{equation}\label{8.2.1}
\begin{split}
&P_{t_1,\ldots,t_k;n}^{\Ewens}(x)\\
&=\frac{t_1^{[x]_{c_1}}t_2^{[x]_{c_2}}\ldots t_k^{[x]_{c_k}}}
{\left(|c_1|t_1+\ldots+|c_k|t_k\right)
\left(|c_1|t_1+\ldots+|c_k|t_k+|G|\right)
\ldots\left(|c_1|t_1+\ldots+|c_k|t_k+(n-1)|G|\right)}.
\end{split}
\nonumber
\end{equation}
It is straightforward to check that the generalized
Chinese restaurant process described above generates the distribution on $G\sim S(n)$ given by the  above formula.
\end{proof}
\begin{figure}[t]
\includegraphics[width=12cm,height=10cm]{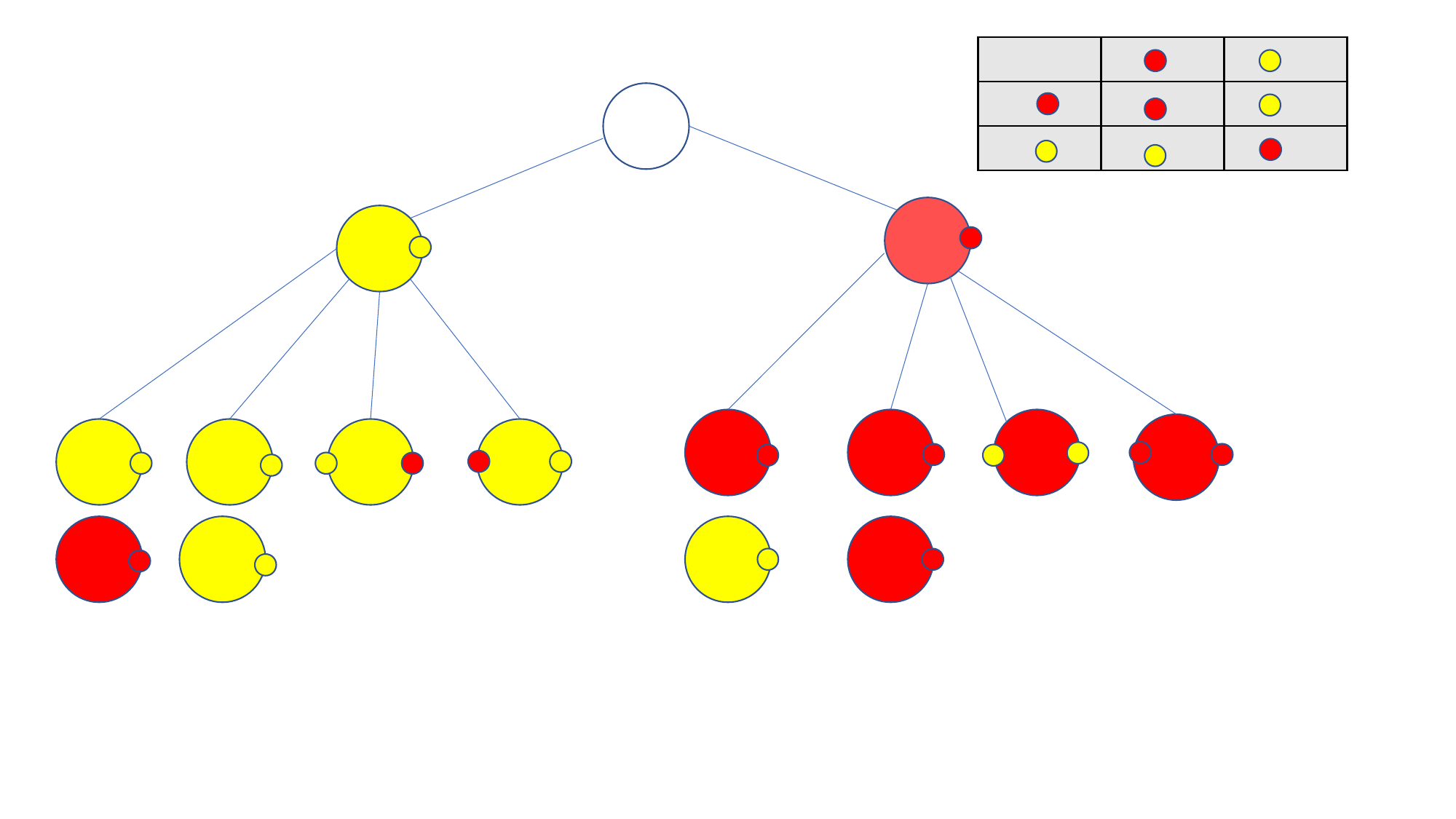}
\caption{An analogue of the Chinese restaurant process.}
\label{FIG1}
\end{figure}
\begin{rem}In order to illustrate the generalized Chinese
restaurant process  constructed above consider the simplest non-trivial
case where $G$ contains just two elements. We represent the unit element of $G$ as a red ball and the other element of $G$ as a yellow ball. The multiplication table of $G$ is shown in Figure \ref{FIG1}.
For each $n$, the elements of $G\sim S(n)$ can be represented by colored cycles (or colored tables) with ball configurations.
The color of a cycle is determined in accordance with the multiplication table: If the multiplication of the balls of a cycle gives a yellow ball, then the color of the cycle is yellow. Otherwise, if the multiplication of the balls in a cycle
gives a red ball, then the color of the cycle is red.  In order to obtain an element of $G\sim S(n)$ from an element of $G\sim S(n-1)$
add a new ball to the configuration describing the chosen element of $G\sim S(n-1)$. The new ball can create a new cycle (with probability
$\frac{t_1}{t_1+t_2+2n}$ provided the new ball is red and with
probability $\frac{t_2}{t_1+t_2+2n}$ provided that the new ball
is yellow). The color of the new cycle will be the same as that of the new ball. Alternatively, the new ball can be placed to the right of the existing ball (with probability $\frac{1}{t_1+t_2+2n}$).
This existing ball is then multiplied by the new ball (see formula (\ref{Element2})), and can change its color according to the multiplication table. Note that the color of the cycle does not change due to the addition of a new ball. Figure \ref{FIG1} shows how the eight elements of $G\sim S(2)$ are obtained from the two elements of $G\sim S(1)$.
\end{rem}
\section{Random set partitions}\label{SectionRandomSetPartitions}
In this Section we discuss random set partitions associated with
the refined Ewens Sampling Formula model. Consider a sample of $n$
genes taken from the population described in Section \ref{SectionModelDefinition}. Label the genes distinctly $1$, $\ldots$, $n$ and assign each an allelic type. The allelic types define
the equivalence relation $\sim$ in the set $\mathbb{N}_n=\{1,\ldots,n\}$.
Thus, we write $i\sim j$ if and only if the genes labeled by $i$ and $j$
have the same allelic type. As a result, the sample of $n$ genes is described by a set partition $\pi$ of $\mathbb{N}_n$, which can be represented in terms of disjoint blocks  as
\begin{equation}\label{3.7.0}
\pi=\left\{B_1^{(r_1)},\ldots,B_{s_1}^{(r_1)},
B_1^{(r_2)},\ldots,B_{s_2}^{(r_2)},\ldots,
B_1^{(r_q)},\ldots,B_{s_q}^{(r_q)}\right\},
\end{equation}
where $r_1$, $\ldots$, $r_q$ are distinct indexes from $\left\{1,\ldots,
k\right\}$ that indicate the classes of alleles in the blocks.
Each $B_j^{(l)}$, $l\in\left\{r_1,\ldots,r_q\right\}$, $j\in\left\{1,\ldots,s_l\right\}$, contains the labels of genes of the $jth$ allele type of the $l$th class. In particular, $B_1^{(r_1)}$ is the block containing the gene labeled by $1$, and $r_1$ specifies the class
of its allelic type, $B_2^{(r_1)}$ contains the smallest element whose allelic type is of the same class $r_1$ that is not in $B_1^{(r_1)}$,
$\ldots$, $B_1^{(r_2)}$ contains the smallest element whose class is different from $r_1$,
$B_2^{(r_2)}$ contains the smallest element whose allelic type is of the same class $r_2$ that is not in $B_1^{(r_2)}$, etc.

If the sample is chosen at random, the corresponding set partition $\Pi_n$ is a random variable. For the model described by the refined
Ewens sampling formula, we obtain
\begin{equation}\label{3.7.1}
\begin{split}
&\Prob\left\{\Pi_n=\left\{B_1^{(r_1)},\ldots,B_{s_1}^{(r_1)},
B_1^{(r_2)},\ldots,B_{s_2}^{(r_2)},\ldots,
B_1^{(r_q)},\ldots,B_{s_q}^{(r_q)}\right\}\right\}\\
&=\frac{(\theta_{r_1})^{s_1}(\theta_{r_2})^{s_2}\ldots(\theta_{r_q})^{s_q}}{(\theta_1+\ldots+\theta_k)_n}\;
\prod\limits_{j_1=1}^{s_1}\left(|B_{j_1}^{(r_1)}|-1\right)!
\ldots
\prod\limits_{j_q=1}^{s_q}\left(|B_{j_q}^{(r_q)}|-1\right)!.
\end{split}
\end{equation}
This follows from the fact that the model is described by the generalized Hoppe urn process; see Section \ref{Section3.2}.

Recall (see, for example, the book by Pitman \cite{PitmanBookCombinatorialProcesses}) that a random set partition is \textit{exchangeable} if
\begin{equation}
\Prob\left\{\Pi_n=\pi\right\}=\Prob\left\{\Pi_n=\sigma(\pi)\right\},
\end{equation}
for each $\sigma\in S(n)$, where $S(n)$ is the permutation group of $\mathbb{N}_n$. Here, $\sigma(\pi)$ is obtained from $\pi$ by relabeling.
That is, $\sigma(\pi)$ is obtained from $\pi$ by putting $\sigma(i)$ and
$\sigma(j)$ in the same block of $\sigma(\pi)$ if and only if $i$ and $j$ are in the same block of $\pi$. The exchangeability of the distribution
(\ref{3.7.1}) is clear since it depends on $\pi$ only through its block sizes.

Let us show that (\ref{3.7.1}) leads to the refined Ewens sampling formula for random multiple partitions of $n$ into $k$ components.
Assume that $\pi$ is a fixed set partition of $\mathbb{N}_n$
defined by (\ref{3.7.0}). Given $\pi$, let $m_j^{(l)}(\pi)$
be the number of blocks $B_1^{(l)}$, $B_2^{(l)}$, $\ldots$ in $\pi$.
Here $l=1,\ldots,k$, and $j=1,\ldots, n$. For each $l$, define the
Young diagram $\lambda^{(l)}(\pi)$ by
\begin{equation}\label{3.8.2}
\lambda^{(l)}(\pi)=\left(1^{m_1^{(l)}(\pi)}2^{m_2^{(l)}(\pi)}\ldots
n^{m_n^{(l)}(\pi)}\right).
\end{equation}
Then $\Lambda_n^{(k)}(\pi)=\left(\lambda^{(1)}(\pi),\ldots,\lambda^{(k)}(\pi)\right)$ is the multiple partition of $n$ into $k$ components associated with the set partition $\pi$ of $\mathbb{N}_n$.
Now, if $\Pi_n$ is a random set partition of $\mathbb{N}_n$ whose distribution is given by equation (\ref{3.7.1}), then the distribution of the corresponding random multiple partition $\mathcal{L}_n^{(k)}\left(\Pi_n\right)$ is given by the refined Ewens sampling formula. In fact, standard combinatorial arguments give us
\begin{equation}\label{3.8.4}
\Prob\left\{\mathcal{L}_n^{(k)}\left(\Pi_n\right)=\Lambda_n^{(k)}\left(\pi\right)\right\}=\frac{n!\Prob\left\{\Pi_n=\pi\right\}}{\prod\limits_{l=1}^k\prod\limits_{j=1}^n
\left(m_j^{(l)}(\pi)\right)!\left(j!\right)^{m_j^{(l)}(\pi)}},
\end{equation}
where $\pi$ is defined by equation (\ref{3.7.0}), and
$\Prob\left\{\Pi_n=\pi\right\}$ is given by (\ref{3.7.1}).
We have
\begin{equation}\label{3.8.5}
\theta_1^{\sum\limits_{j=1}^nm_j^{(1)}(\pi)}\ldots
\theta_k^{\sum\limits_{j=1}^nm_j^{(k)}(\pi)}
=(\theta_{r_1})^{s_1}(\theta_{r_2})^{s_2}\ldots(\theta_{r_q})^{s_q},
\end{equation}
and
\begin{equation}\label{3.8.6}
\prod\limits_{l=1}^k\prod\limits_{j=1}^n
\left[(j-1)!\right]^{m_j^{(l)}(\pi)}
=\prod\limits_{j_1=1}^{s_1}\left(|B_{j_1}^{(r_1)}|-1\right)!
\ldots\prod\limits_{j_q=1}^{s_q}\left(|B_{j_q}^{(r_q)}|-1\right)!.
\end{equation}
Taking into account (\ref{3.8.4})-(\ref{3.8.6}) we find
\begin{equation}
\Prob\left\{\mathcal{L}_n^{(k)}\left(\Pi_n\right)=\Lambda_n^{(k)}(\pi)\right\}
=\frac{\theta_1^{\sum\limits_{j=1}^nm_j^{(1)}(\pi)}\ldots
\theta_k^{\sum\limits_{j=1}^nm_j^{(k)}(\pi)}}{\left(\theta_1+\ldots+\theta_k\right)_n}
\frac{n!}{\prod\limits_{l=1}^k\prod\limits_{j=1}^n
\left(m_j^{(l)}(\pi)\right)!j^{m_j^{(l)}(\pi)}},
\end{equation}
which is the refined Ewens sampling formula.

Let $\left(\Pi_n\right)_{n=1}^{\infty}$ be a sequence of exchangeable
random set partitions. For $1\leq m\leq n$ denote by $\Pi_{m,n}$
the restriction of $\Pi_n$ to $\mathbb{N}_m=\{1,\ldots,m\}$.
The sequence $\left(\Pi_n\right)_{n=1}^{\infty}$ is \textit{consistent in distribution} if $\Pi_m$ has the same distribution as $\Pi_{m,n}$
for every $m<n$. It is not difficult to see that the sequence $\left(\Pi_n\right)_{n=1}^{\infty}$ of random set partitions associated
with random samples whose distribution is defined by (\ref{3.7.1})
is consistent in distribution. This follows from (\ref{3.7.1}), and from the relation
\begin{equation}
\begin{split}
&\frac{1}{\left(|B_1^{(r_1)}|+\ldots+|B_{s_1}^{(r_1)}|
+\ldots+|B_1^{(r_q)}|+\ldots+|B_{s_q}^{(r_q)}|\right)}
+\frac{\theta_1+\ldots+\theta_k}{\left(\theta_1+\ldots+\theta_k\right)_{n+1}}\\
&=\frac{\theta_1+\ldots+\theta_k+n}{\left(\theta_1+\ldots+\theta_k\right)_{n+1}}
=\frac{1}{\left(\theta_1+\ldots+\theta_k\right)_{n}}.
\end{split}
\end{equation}
Recall that a sequence $\Pi_{\infty}=\left(\Pi_n\right)_{n=1}^{\infty}$
of exchangeable random set partitions consistent in distribution
is called an \textit{infinite exchangeable random set partition},
see, for example, Pitman \cite{PitmanBookCombinatorialProcesses},
Section 2. In particular, the sequence $\left(\Pi_n\right)_{n=1}^{\infty}$ where distribution of $\Pi_n$
is defined by equation (\ref{3.7.1}) is an infinite exchangeable set partition. Theorem 2.2 in Pitman \cite{PitmanBookCombinatorialProcesses}
describes a bijection (the Kingman correspondence) between infinite
exchangeable random set partitions and probability distributions
on the set
$$
\nabla=\left\{\left(p_1,p_2,\ldots\right):\;
p_1\geq p_2\geq\ldots\geq 0\;\;\mbox{and}\;\;\sum\limits_{i=1}^{\infty}p_i\leq 1\right\}.
$$
Note that the coordinates $x^{(l)}$ of $\overline{\nabla}_0^{(k)}$ in Proposition \ref{PropositionMEwenSASIntegral} satisfy
$$
\sum\limits_{l=1}^k\sum\limits_{i=1}^{\infty}x_i^{(l)}=1.
$$
Therefore,  $\overline{\nabla}_0^{(k)}$ can be embedded into
$\nabla$. Under this embedding the probability measure $PD\left(\theta_1,\ldots,\theta_k\right)$ on   $\overline{\nabla}_0^{(k)}$ will induce a certain probability measure
on $\nabla$. In this sense Proposition  \ref{PropositionMEwenSASIntegral}
is a manifestation of the Kingman correspondence.

\section{Numbers of alleles}\label{SectionNumbersAllelicTypes}

Let $K_n^{(l)}$, $l\in\{1,\ldots,k\}$, be the number of different alleles of class $l$ in a random sample of size $n$.
The purpose of this Section is to derive exact formulae
for the random variables $K_n^{(l)}$, and to describe their
limiting distributions in different asymptotic regimes under the assumption that the allelic composition of the sample
is governed by the refined Ewens sampling formula, equation (\ref{RefinedESF}).
\subsection{Exact formulas for $K_n^{(l)}$}\label{SectionExactFormulasK}
In the context of alleles, the sampling of a black object of the $l$th class in the generalized Hoppe urn described in Section
\ref{SectionProof}
is a mutation characterized by the parameter $\theta_l$ (where $l\in\{1,\ldots,k\}$), which induces new allele types of class $l$. The selection of a non-black object of the $l$th class has the interpretation that the genes are multiplying themselves within their allelic type. With this interpretation,
we have
\begin{equation}\label{K3.1}
K_n^{(l)}=Y_1^{(l)}+Y_2^{(l)}+\ldots+Y_n^{(l)},
\end{equation}
where $Y_j^{(l)}$ is the indicator that assumes $1$ if the black object of class $l$  (with mass $\theta_l>0$) is chosen in the $j$th draw, and assumes $0$ otherwise. Then we can write
\begin{equation}\label{K3.2}
\Prob\left(Y_j^{(l)}=1\right)=\frac{\theta_l}{\theta_1+\ldots+\theta_k+j-1},
\;\;\Prob\left(Y_j^{(l)}=0\right)=1-\frac{\theta_l}{\theta_1+\ldots+\theta_k+j-1}.
\end{equation}
Note that the Bernoulli random variables $Y_1^{(l)}$, $\ldots$,$Y_n^{(l)}$ are independent. Taking expectations, we find
\begin{equation}
\mathbb{E}\left(K_n^{(l)}\right)
=\mathbb{E}(Y_1^{(l)})+\mathbb{E}(Y_2^{(l)})+\ldots+\mathbb{E}(Y_n^{(l)}),
\end{equation}
which gives
\begin{equation}\label{ExpectationNumberAlleles}
\mathbb{E}\left(K_n^{(l)}\right)=\frac{\theta_l}{\theta_1+\ldots+\theta_k}
+\frac{\theta_l}{\theta_1+\ldots+\theta_k+1}+\ldots+
\frac{\theta_l}{\theta_1+\ldots+\theta_k+n-1}.
\end{equation}
Denote by $\mathcal{H}_n^{(p)}(x)$  a function defined by
\begin{equation}\label{H3.5}
\mathcal{H}_n^{(p)}(x)=\frac{1}{x^p}+\frac{1}{(x+1)^p}+\ldots+\frac{1}{(x+n-1)^p},
\end{equation}
where $p\in \{1,2,\ldots\}$. Then we obtain
\begin{equation}\label{3.6}
\mathbb{E}\left(K_n^{(l)}\right)=\theta_l\mathcal{H}_n^{(1)}\left(\theta_1+\ldots+\theta_k\right).
\end{equation}
The independence of Bernoulli random variables $Y_1^{(l)}$, $\ldots$,$Y_n^{(l)}$ implies
\begin{equation}\label{3.7}
\Var\left(K_n^{(l)}\right)
=\Var(Y_1^{(l)})+\Var(Y_2^{(l)})+\ldots+\Var(Y_n^{(l)}),
\end{equation}
which gives
\begin{equation}
\Var\left(K_n^{(l)}\right)=\sum\limits_{j=1}^n\frac{\theta_l}{\theta_1+\ldots+\theta_k+j-1}\left(1-\frac{\theta_l}{\theta_1+\ldots+\theta_k+j-1}\right).
\end{equation}
Therefore,
\begin{equation}\label{3.9}
\Var\left(K_n^{(l)}\right)=\theta_l\mathcal{H}_n^{(1)}\left(\theta_1+\ldots+\theta_k\right)-\theta_l^2\mathcal{H}_n^{(2)}\left(\theta_1+\ldots+\theta_k\right).
\end{equation}
Denote by $\left[\begin{array}{c}
  n \\
  m
\end{array}\right]$ be the number of permutations from $S(n)$ with $m$ disjoint cycles.
This number can be written explicitly as\footnote{Notation $\lambda\vdash n$ means that
the Young diagram $\lambda$ contains $n$ boxes.}
$$
\left[\begin{array}{c}
  n \\
  m
\end{array}\right]=\underset{l(\lambda)=m}{\sum\limits_{\lambda\vdash n}}
\frac{n!}{\prod_{j=1}^nj^{m_j(\lambda)}\left(m_j(\lambda)\right)!}.
$$
With this notation we can write the joint distribution of the random variables
$K_n^{(1)}$, $\ldots$, $K_n^{(k)}$ as
\begin{equation}\label{JointDistributionAllelicTypes}
\Prob\left(K_n^{(1)}=p_1,\ldots,K_n^{(k)}=p_k\right)=
\frac{n!\theta_1^{p_1}\ldots\theta_k^{p_k}}{\left(\theta_1+\ldots+\theta_k\right)_n}
\sum\limits_{n_1+\ldots+n_k=n}\frac{\left[\begin{array}{c}
  n_1 \\
  p_1
\end{array}\right]\ldots
\left[\begin{array}{c}
  n_k \\
  p_k
\end{array}\right]}{n_1!\ldots n_k!}.
\end{equation}
In the case of $k=1$ (all alleles belong to the same class)
formula (\ref{ExpectationNumberAlleles}) for the expectation of $K_n^{(l)}$, formula (\ref{3.9}) for the variance of $K_n^{(l)}$,
and formula (\ref{JointDistributionAllelicTypes})
for the joint distribution of random variables $K_n^{(1)}$,
$\ldots$, $K_n^{(k)}$ are reduced
to the well-known formulae derived by Ewens in Ref. \cite{Ewens}.
\subsection{Inequalities for $\mathcal{H}_n^{(1)}(x)$ and
$\mathcal{H}_n^{(2)}(x)$}\label{SectionInequalitiesHfunctions}
In the subsequent asymptotic analysis we will need
the following inequalities for the functions  $\mathcal{H}_n^{(1)}(x)$ and
$\mathcal{H}_n^{(2)}(x)$ defined by equation (\ref{H3.5}).
\begin{lem}\label{4.1.1}
Let the functions  $\mathcal{H}_n^{(1)}(x)$ and
$\mathcal{H}_n^{(2)}(x)$ be defined by equation (\ref{H3.5}). Then the following inequalities hold true
\begin{equation}\label{4.1.1in}
\frac{n}{2x(x+n)}<\mathcal{H}_n^{(1)}(x)-\log\left(1+\frac{n}{x}\right)<\frac{n}{x(x+n)},
\end{equation}
and
\begin{equation}\label{4.1.2}
\frac{n}{x(x+n)}<\mathcal{H}_n^{(2)}(x)<\frac{1}{x^2}+\frac{n-1}{x(x+n-1)}.
\end{equation}
Here $x>0$, and $n=1,2,\ldots$.
\end{lem}
\begin{proof}The proof of these inequalities can be extracted
from that of Proposition 1 in Tsukuda \cite{Tsukuda}. We reproduce
the proof here for the reader's convenience.

In order to prove (\ref{4.1.1in}) note that
$$
\mathcal{H}_n^{(1)}(x)-\log\left(1+\frac{n}{x}\right)=\sum\limits_{j=1}^n
\left(\frac{1}{x+j-1}-\int\limits_{x+j-1}^{x+j}\frac{dt}{t}\right).
$$
Since
$$
\frac{1}{x+j}<\int\limits_{x+j-1}^{x+j}\frac{dt}{t}<\frac{1}{x+j-1},\;\;\;\;
\int\limits_{x+j-1}^{x+j}\frac{dt}{t}<\frac{1}{2}\left(\frac{1}{x+j-1}+\frac{1}{x+j}\right),
$$
we have
$$
\frac{1}{2}\left(\frac{1}{x+j-1}-\frac{1}{x+j}\right)<\frac{1}{x+j-1}-\int\limits_{x+j-1}^{x+j}\frac{dt}{t}<\frac{1}{x+j-1}-\frac{1}{x+j}.
$$
This gives
$$
\frac{1}{2}\sum\limits_{j=1}^n\left(\frac{1}{x+j-1}-\frac{1}{x+j}\right)<
\mathcal{H}_n^{(1)}(x)-\log\left(1+\frac{n}{x}\right)
<\sum\limits_{j=1}^n\left(\frac{1}{x+j-1}-\frac{1}{x+j}\right),
$$
and inequality (\ref{4.1.1in}) follows. In order to prove (\ref{4.1.2}) note that
$$
\sum\limits_{j=1}^n\frac{1}{\left(x+j-1\right)^2}>\int\limits_x^{x+n}\frac{dt}{t^2}=\frac{n}{x(x+n)},
$$
and that
$$
\frac{n-1}{x(x+n-1)}=\int\limits_x^{x+n-1}\frac{dt}{t^2}>\frac{1}{(x+1)^2}+\ldots+
\frac{1}{(x+n-1)^2}=\sum\limits_{j=1}^n\frac{1}{(x+j-1)^2}-\frac{1}{x^2}.
$$
\end{proof}

\subsection{Limit theorems}
\subsubsection{The case of constant mutation parameters}
First, let us assume that the mutation parameters
$\theta_1$, $\ldots$, $\theta_k$ are constant strictly positive numbers that do not depend on the sample size $n$. Then the following theorem
can be easily obtained from the results of Sections \ref{SectionExactFormulasK} and \ref{SectionInequalitiesHfunctions}.
\begin{thm}\label{TheoremKconstant}Let $K_n^{(l)}$, $l\in\{1,\ldots,k\}$, be the number of different alleles of class $l$ in a random sample of size $n$, and assume that the mutation rates $\theta_1$, $\ldots$, $\theta_k$ are constant strictly positive numbers. Then
for each $l$, $l\in\left\{1,\ldots,k\right\}$,
we have
\begin{equation}\label{4.1.4}
\mathbb{E}\left(K_n^{(l)}\right)=\theta_l\log(n)+O(1),
\end{equation}
\begin{equation}\label{4.1.5}
\Var\left(K_n^{(l)}\right)=\theta_l\log(n)+O(1),
\end{equation}
\begin{equation}\label{4.1.6}
\frac{K_n^{(l)}}{\log(n)}\overset{P}{\longrightarrow}\theta_l,
\end{equation}
and
\begin{equation}\label{4.1.7}
\frac{K_n^{(l)}-\theta_l\log(n)}{\sqrt{\log(n)}}\overset{D}{\longrightarrow}\mathcal{N}(0,1),
\end{equation}
as $n\rightarrow\infty$. In (\ref{4.1.6}) the convergence is in probability, in (\ref{4.1.7}) the convergence is in distribution.
We have denoted by $\mathcal{N}(0,1)$ a standard normal random variable.
\end{thm}
\begin{proof} Formulae (\ref{4.1.4}) and (\ref{4.1.5}) follow from
(\ref{3.6}), (\ref{3.7}), (\ref{4.1.1}) and (\ref{4.1.2}). The proof of (\ref{4.1.6})
and (\ref{4.1.7}) is a repetition of arguments for the standard Ewens distribution; see the book by Mahmoud \cite{Mahmoud}, Section 9.
\end{proof}
\begin{rem}The asymptotic normality of $K_n^{(l)}$ (equation (\ref{4.1.7}))
is an extension of the Watterson limit theorem \cite{Watterson} to the case
of an arbitrary number $k$ of classes of alleles.
\end{rem}
\subsubsection{Growth of mutation parameters with $n$}
It is of interest to investigate asymptotic properties of numbers of  $K_n^{(l)}$ in the situation where the mutation
parameters $\theta_1$, $\ldots$, $\theta_k$ tend to infinity
as the size $n$ of a sample increases. In the case of $k=1$
(that is, in the case of the classical Ewens formula), different asymptotic regimes where both $n$ and the mutation parameter tend to infinity were considered in Feng \cite{FengL}, Tsukuda \cite{Tsukuda, Tsukuda1}. Our aim here is to extend some results obtained in Ref. \cite{Tsukuda} to the situation where the allelic composition of a sample is described by the refined Ewens sampling formula, equation (\ref{RefinedESF}).
\begin{prop}\label{Proposition12.1}
Let $K_n^{(l)}$, $l\in\{1,\ldots,k\}$, be the number of different alleles of class $l$ in a random sample of size $n$. If the mutation parameters $\theta_1$, $\ldots$, $\theta_k$  are strictly positive valued functions of $n$
growing to infinity as $n\rightarrow\infty$, then
\begin{equation}\label{12.2}
\mathbb{E}\left(K_n^{(l)}\right)=\theta_l\log\left(1+\frac{n}{w}\right)+O(1),
\end{equation}
and
\begin{equation}\label{12.3}
\Var\left(K_n^{(l)}\right)=\theta_l\log\left(1+\frac{n}{w}\right)
-\frac{n\theta_l^2}{w(w+n)}+O(1),
\end{equation}
as $n\rightarrow\infty$, where $w$ is the sum of the mutation parameters,
\begin{equation}\label{smalldablu}
w=\theta_1+\ldots+\theta_k.
\end{equation}
\end{prop}
\begin{proof}
Formula (\ref{3.6}) and inequality (\ref{4.1.1in}) imply
\begin{equation}\label{12.4}
\frac{\theta_ln}{2w(w+n)}<\mathbb{E}\left(K_n^{(l)}\right)-\theta_l\log\left(1+\frac{n}{w}\right)<\frac{\theta_ln}{w(w+n)},
\end{equation}
which gives equation (\ref{12.2}). From inequality (\ref{4.1.2}) we conclude that
\begin{equation}\label{12.5}
0<\theta_l^2\mathcal{H}_n^{(2)}(w)-\frac{n\theta^2_l}{w(w+n)}<\frac{\theta_l^2}{w^2}
-\frac{\theta_l^2}{(w+n)(w+n-1)}.
\end{equation}
Therefore, equation (\ref{12.3}) follows from (\ref{3.6}),(\ref{3.9}), (\ref{12.2}),  and (\ref{12.5}).
\end{proof}
Next, we investigate the convergence in probability
of random variables $K_n^{(l)}$, $l\in\left\{1,\ldots,k\right\}$.
\begin{thm}\label{TheoremConvergenceInProbability}Assume that the mutation parameters $\theta_1$,
$\ldots$, $\theta_k$ depend on $n$ as
\begin{equation}\label{Thetan}
\theta_1(n)=\alpha_1n^{\beta},\ldots,\theta_k(n)=\alpha_kn^{\beta};\;\;\beta\geq 0;\;\;\alpha_1>0,\ldots,\alpha_k>0.
\end{equation}
Let $K_n^{(l)}$, $l\in\{1,\ldots,k\}$, be the number of different alleles of class $l$ in a random sample of size $n$.\\
(a) If $0\leq \beta<1$, then $$\frac{K_n^{(l)}}{n^{\beta}\log(n)}\overset{P}{\longrightarrow}\alpha_l(1-\beta).$$\\
(b) If $\beta=1$, then $$\frac{K_n^{(l)}}{n}\overset{P}{\longrightarrow}\alpha_l\log\frac{1+A}{A},$$ where
\begin{equation}\label{Aalpha}
A=\alpha_1+\ldots+\alpha_k.
\end{equation}
(c) If $\beta>1$, then
$$
\frac{K_n^{(l)}}{n}\overset{P}{\longrightarrow}\frac{\alpha_l}{A},
$$
where $A$ is defined by equation (\ref{Aalpha}).
\end{thm}
\begin{proof}If the dependence of the mutation parameters on $n$ is given by equation (\ref{Thetan}), then the following asymptotic expressions for $\mathbb{E}\left(K_n^{(l)}\right)$ and $\Var\left(K_n^{(l)}\right)$ can be obtained from equations (\ref{12.2}) and (\ref{12.3}).\\
(a) If $0\leq\beta<1$, then
$$
\mathbb{E}\left(K_n^{(l)}\right)=\alpha_l(1-\beta)n^{\beta}\log(n)+O\left(n^{\beta}\right),
$$
and
$$
\Var\left(K_n^{(l)}\right)=\alpha_l(1-\beta)n^{\beta}\log(n)+O\left(n^{\beta}\right),
$$
as $n\rightarrow\infty$.\\
(b) If $\beta=1$, then
$$
\mathbb{E}\left(K_n^{(l)}\right)=\alpha_ln
\log\left(\frac{1+A}{A}\right)+O\left(1\right),
$$
where $A$ is defined by equation (\ref{Aalpha}), and
$$
\Var\left(K_n^{(l)}\right)=\alpha_ln
\log\left(\frac{1+A}{A}\right)-\frac{\alpha_l^2n}{A(1+A)}+O\left(1\right),
$$
as $n\rightarrow\infty$.\\
(c) If $\beta>1$, then
$$
\mathbb{E}\left(K_n^{(l)}\right)=\frac{\alpha_l}{A}n
+O\left(n^{2-\beta}\right),
$$
and
$$
\Var\left(K_n^{(l)}\right)=\frac{\alpha_l}{A}\left(1-\frac{\alpha_l}{A}\right)n
+O\left(n^{2-\beta}\right),
$$
as $n\rightarrow\infty$. Note that for each $\beta\geq 0$ the condition
\begin{equation}\label{Cheb1}
\frac{\Var\left(K_n^{(l)}\right)}{\left(\mathbb{E}\left(K_n^{(l)}\right)\right)^2}
\longrightarrow 0,\;\;\mbox{as}\;\; n\longrightarrow\infty,
\end{equation}
is satisfied. Recall that according to Chebyshev's inequality, for any fixed $\epsilon>0$
$$
\Prob\left(\left|K_n^{(l)}-\mathbb{E}\left(K_n^{(l)}\right)\right|>\epsilon\right)\leq
\frac{\Var\left(K_n^{(l)}\right)}{\epsilon^2}.
$$
Replacing $\epsilon$ by $\epsilon\mathbb{E}\left(K_n^{(l)}\right)$, we obtain
\begin{equation}\label{Cheb3}
\Prob\left(\left|\frac{K_n^{(l)}}{\mathbb{E}\left(K_n^{(l)}\right)}-1\right|>\epsilon\right)\leq
\frac{\Var\left(K_n^{(l)}\right)}{\epsilon^2\left(\mathbb{E}\left(K_n^{(l)}\right)\right)^2}.
\end{equation}
It follows from (\ref{Cheb1}) and (\ref{Cheb3}) that
\begin{equation}\label{Cheb4}
\frac{K_n^{(l)}}{\mathbb{E}\left(K_n^{(l)}\right)}\overset{P}{\longrightarrow} 1,
\end{equation}
as $n\rightarrow\infty$, for each $\beta\geq 0$. The results
in the statement of the Theorem follow from (\ref{Cheb4}), and from the convergence relations
\begin{equation}
\frac{\mathbb{E}\left(K_n^{(l)}\right)}{n^{\beta}\log(n)}
\longrightarrow \alpha_l(1-\beta),\;\;\;  0\leq\beta<1,
\end{equation}
\begin{equation}
\frac{\mathbb{E}\left(K_n^{(l)}\right)}{n}
\longrightarrow \alpha_l\log\left(\frac{1+A}{A}\right),\;\;\;  \beta=1,
\end{equation}
and
\begin{equation}
\frac{\mathbb{E}\left(K_n^{(l)}\right)}{n}
\longrightarrow \frac{\alpha_l}{A},\;\;\;  \beta>1,
\end{equation}
as $n\rightarrow\infty$.
\end{proof}
\begin{rem}
(a) Equation (\ref{Cheb1}) leads to the fact that $K_n^{(l)}$
is highly concentrated around its expectation in the case where the mutation parameters depend on $n$ as in equation
(\ref{Thetan}). However, if the mutation parameters are allowed to grow at different rates, then (\ref{Cheb1})
might not be true. For example, if $k=2$,
$\theta_1(n)=\alpha_1n^{1/2}$, and $\theta_2(n)=\alpha_2n^{3/2}$, then we obtain
$\mathbb{E}\left(K_n^{(l)}\right)=O(1)$, $\Var\left(K_n^{(l)}\right)=O(1)$, and (\ref{Cheb1})
is not valid.\\
(b)
If $k=1$, the asymptotic normality of $K_n^{(l)}$
as $n\rightarrow\infty$ (with a dependence of the mutation parameter on $n$) was established in Tsukuda \cite{Tsukuda}.
Theorem \ref{Theorem12.8} below presents the asymptotic normality of $K_n^{(l)}$ when the number $k$ of allele classes is arbitrary and the mutation parameters
$\theta_1$, $\ldots$, $\theta_k$ depend on the sample size $n$
as in equation (\ref{Thetan}).\\
(c) One can investigate the number of distinct values in a random sample $X_1$, $\ldots$, $X_n$ from the Dirichlet process, see Antoniak \cite{Antoniak}, Korwar and Hollander \cite{KorwarHollander}, and the book by  Ghosal and van der Vaart \cite{GhosalvanderVaart},
Section 4.1.5 for a detailed account and  further references. For $i\in\mathbb{N}$ define
$D_i=1$ if $X_i$ does not belong to $\left\{X_1,\ldots,X_{i-1}\right\}$, and set $D_i=0$ otherwise.
Then $K_n=\sum\limits_{i=1}^nD_i$ is the number of distinct values among the first $n$ observables.
If $k=1$, then the results of Theorem \ref{TheoremKconstant} can be understood as those for $K_n$
obtained in Proposition 4.8 of Ghosal and van der Vaart \cite{GhosalvanderVaart},
Section 4.1.5.  Furthermore, in the context of the Dirichlet process, the parameter $\theta$ of the classical Ewens formula is replaced by the total mass $M$ of the base measure, which can be small.  Ghosal and van der Vaart \cite{GhosalvanderVaart},
Section 4.1.5 give inequalities for $\mathbb{E}\left(K_n\right)$ and $\Var\left(K_n\right)$ valid for all $M$ and $n$.
\end{rem}

\begin{thm}\label{Theorem12.8}Let $K_n^{(l)}$, $l\in\{1,\ldots,k\}$, be the number of different alleles of class $l$ in a random sample of size $n$,
and assume that the mutation parameters  are given by equation (\ref{Thetan}).\\
(a) If $0<\beta\leq\frac{3}{2}$, then for each $l\in\left\{1,\ldots,k\right\}$
\begin{equation}
\frac{K_n^{(l)}-\theta_l\log\left(1+\frac{n}{w}\right)}{\sqrt{\theta_l\log\left(1+\frac{n}{w}\right)-
\frac{n\theta_l^2}{w
\left(w+n\right)}}}
\overset{D}{\longrightarrow}\mathcal{N}\left(0,1\right),
\end{equation}
as $n\rightarrow\infty$, where $w$ is the sum of the mutation parameters defined by equation (\ref{smalldablu}), and $\mathcal{N}(0,1)$ denotes a standard normal random variable.\\
(b) If $\beta>\frac{3}{2}$ and $k\geq 2$, then
\begin{equation}\label{11.3}
\frac{K_n^{(l)}-n\frac{\theta_l}{w}}{\sqrt{n\frac{\theta_l}{w}
\left(1-\frac{\theta_l}{w}\right)}}
\overset{D}{\longrightarrow}\mathcal{N}(0,1)
\end{equation}
is satisfied for each $l\in\left\{1,\ldots,k\right\}$, as $n\rightarrow\infty$.
\end{thm}
\begin{rem}
(a) If $k=1$, then $\frac{\theta_l}{w}=1$, and
(\ref{11.3}) does not hold true. For the results on $K_n^{(l)}$ in the case $k=1$ we refer the reader
to Theorem 2 and Theorem 3 in Tsukuda \cite{Tsukuda}.\\
(b) If $\beta=0$, then the convergence in distribution
to a standard normal random variable is given by equation
(\ref{4.1.7}).\\
(c) If $k=1$, and $n^2/\theta_1\rightarrow c>0$ as $n\rightarrow\infty$, then Tsukuda \cite{Tsukuda} showed that
\begin{equation}
\frac{K_n^{(l)}-\theta_1\log\left(1+\frac{n}{\theta_1}\right)}{\sqrt{\theta_1\left(\log\left(1+\frac{n}{\theta_1}\right)-\frac{n}{n+\theta_1}\right)}}\overset{D}{\longrightarrow}
\frac{c/2-P_{c/2}}{\sqrt{c/2}},
\end{equation}
where $P_{c/2}$ is a Poisson variable with mean $c/2$. However, Theorem \ref{Theorem12.8}, (b) says that for $k\geq 2$
we still have convergence in distribution to a standard normal random variable.
\end{rem}
\begin{proof}
The proof here is an extension
of that in Tsukuda \cite{Tsukuda}, Section 3.2 to the case of an arbitrary number $k$ of mutation parameters.\\
(a) Assume that $0<\beta\leq\frac{3}{2}$.
Set
\begin{equation}\label{12.9}
\Delta_l=\theta_l\log\left(1+\frac{n}{w}\right),\;\;\sigma_l^2=\theta_l\log\left(1+\frac{n}{w}\right)
-\frac{n\theta_l^2}{w(w+n)}.
\end{equation}
We note that
\begin{equation}\label{12.10}
\Delta_l=\sigma_l^2+\frac{n\theta_l^2}{w(w+n)}.
\end{equation}

Using (\ref{K3.1}) and independence of the Bernoulli random variables
$Y_1^{(l)}$, $\ldots$, $Y_n^{(l)}$ we find
\begin{equation}\label{4.2.11}
\mathbb{E}\left(e^{\frac{K_n^{(l)}-\Delta_l}{\sigma_l}t}\right)=
e^{-\frac{\Delta_l}{\sigma_l}t}\frac{\Gamma\left(\theta_l\left(e^{\frac{t}{\sigma_l}}-1\right)+w+n\right)}{\Gamma\left(w+n\right)}
\frac{\Gamma(w)}{\Gamma\left(\theta_l\left(e^{\frac{t}{\sigma_l}}-1\right)+w\right)}.
\end{equation}
Application of the Stirling formula gives
\begin{equation}\label{4.2.12}
\begin{split}
&\frac{\Gamma\left(\theta_l\left(e^{\frac{t}{\sigma_l}}-1\right)+w+n\right)}{\Gamma\left(w+n\right)}
\frac{\Gamma(w)}{\Gamma\left(\theta_l\left(e^{\frac{t}{\sigma_l}}-1\right)+w\right)}\\
&\sim\frac{\left(\theta_l\left(e^{\frac{t}{\sigma_l}}-1\right)+w+n\right)^{\theta_l\left(e^{\frac{t}{\sigma_l}}-1\right)+w+n-1}}{(w+n)^{w+n-1}}\;
\frac{w^{w-1}}{\left(\theta_l\left(e^{\frac{t}{\sigma_l}}-1\right)+w\right)^{\theta_l\left(e^{\frac{t}{\sigma_l}}-1\right)+w-1}},
\end{split}
\end{equation}
as $n\rightarrow\infty$.
The right-hand side of  (\ref{4.2.12}) can be rewritten as
\begin{equation}\label{12.11}
\begin{split}
&\left[1+\frac{n}{w}\right]^{\theta_l\left(e^{t/\sigma_l}-1\right)}
\left[1+\frac{\theta_l\left(e^{t/\sigma_l}-1\right)}{w+n}\right]
^{\theta_l\left(e^{t/\sigma_l}-1\right)+w+n-1}\\
&\times
\left[1+\frac{\theta_l}{w}\left(e^{t/\sigma_l}-1\right)\right]
^{-\theta_l\left(e^{t/\sigma_l}-1\right)-w+1}.
\end{split}
\end{equation}
The equations (\ref{4.2.11}), (\ref{4.2.12}), (\ref{12.9})-(\ref{12.11}) imply
\begin{equation}\label{12.12}
\begin{split}
&\log\left[\mathbb{E}\left(e^{\frac{K_n^{(l)}-\Delta_l}{\sigma_l}t}\right)\right]
\sim-\sigma_l t-\frac{n\theta_l^2t}{\sigma_l w(w+n)}
+\sigma_l^2\left(e^{t/\sigma_l}-1\right)\\
&+\frac{n\theta_l^2\left(e^{t/\sigma_l}-1\right)}{w(w+n)}
+\left[\theta_l\left(e^{t/\sigma_l}-1\right)+n-1+w\right]
\log\left[1+\frac{\theta_l\left(e^{t/\sigma_l}-1\right)}{w+n}\right]\\
&-\left[\theta_l\left(e^{t/\sigma_l}-1\right)+w-1\right]
\log\left[1+\frac{\theta_l}{w}\left(e^{t/\sigma_l}-1\right)\right].\\
&=-\sigma_l t-\frac{n\theta_l^2t}{\sigma_l w(w+n)}
+\sigma_l^2\left(e^{t/\sigma_l}-1\right)\\
&+\frac{n\theta_l^2\left(e^{t/\sigma_l}-1\right)}{w(w+n)}
+\left[\theta_l\left(e^{t/\sigma_l}-1\right)+n+w\right]
\log\left[1+\frac{\theta_l\left(e^{t/\sigma_l}-1\right)}{w+n}\right]\\
&-\left[\theta_l\left(e^{t/\sigma_l}-1\right)+w\right]
\log\left[1+\frac{\theta_l}{w}\left(e^{t/\sigma_l}-1\right)\right]+o(1),
\end{split}
\end{equation}
as $n\rightarrow\infty$.

Let us consider the case where $0<\beta<3/2$.
In the following, it is important that if $0<\beta<3/2$, then $\theta_l/\sigma_l^3\rightarrow 0$
as $n\rightarrow\infty$. This follows from equation (\ref{12.9}), from
(\ref{12.3}), and from explicit asymptotic formulae for the variance of $K_n^{(l)}$ in the proof of Theorem \ref{TheoremConvergenceInProbability}.
We can write
\begin{equation}\label{12.13}
\begin{split}
&\left[\theta_l\left(e^{t/\sigma_l}-1\right)+n+w\right]
\log\left[1+\frac{\theta_l\left(e^{t/\sigma_l}-1\right)}{w+n}\right]\\
&=\left[\theta_l\left(e^{t/\sigma_l}-1\right)+n+w\right]
\frac{\theta_l\left(e^{t/\sigma_l}-1\right)}{w+n}-\left[\theta_l\left(e^{t/\sigma_l}-1\right)+n+w\right]
\frac{\theta_l^2\left(e^{t/\sigma_l}-1\right)^2}{2(w+n)^2}+o(1)\\
&=\theta_l\left(e^{t/\sigma_l}-1\right)+\frac{1}{2}\frac{\theta_l^2\left(e^{t/\sigma_l}-1\right)^2}{w+n}+o(1),
\end{split}
\end{equation}
since $\frac{\theta_l^3(e^{t/\sigma_l}-1)^3}{2(w+n)^2}=o(1)$
provided that $\theta_l/\sigma_l^3\rightarrow 0$
as $n\rightarrow\infty$.
Also,
\begin{equation}\label{12.14}
\begin{split}
&-\left[\theta_l\left(e^{t/\sigma_l}-1\right)+w\right]
\log\left[1+\frac{\theta_l\left(e^{t/\sigma_l}-1\right)}{w}\right]\\
&=-\theta_l\left(e^{t/\sigma_l}-1\right)-\frac{1}{2}\frac{\theta_l^2\left(e^{t/\sigma_l}-1\right)^2}{w}+o(1),
\end{split}
\end{equation}
as $n\rightarrow\infty$. Taking (\ref{12.13}), (\ref{12.14})
into account we rewrite (\ref{12.12}) as
\begin{equation}\label{12.15}
\begin{split}
&\log\left[\mathbb{E}\left(e^{\frac{K_n^{(l)}-\Delta_l}{\sigma_l}t}\right)\right]
\sim-\sigma_l t
+\sigma_l^2\left(e^{t/\sigma_l}-1\right)-\frac{n\theta_l^2t}{\sigma_l w(w+n)}\\
&+\frac{n\theta^2_l\left(e^{t/\sigma_l}-1\right)}{w(w+n)}
-\frac{n\theta^2_l\left(e^{t/\sigma_l}-1\right)^2}{2w(w+n)}.
\end{split}
\end{equation}
The first two terms in the right-hand side of equation (\ref{12.15}) give
\begin{equation}
-\sigma_l t
+\sigma_l^2\left(e^{t/\sigma_l}-1\right)=\frac{t^2}{2}+o(1),
\end{equation}
as $n\rightarrow\infty$. It is easy to check that the last three terms on the right-hand side of equation (\ref{12.15}) are equal to $o(1)$, as $n\rightarrow\infty$. We conclude that
\begin{equation}\label{LogGeneration}
\log\left[\mathbb{E}\left(e^{\frac{K_n^{(l)}-\Delta_l}{\sigma_l}t}\right)\right]
\sim \frac{t^2}{2},
\end{equation}
as $n\rightarrow\infty$. This implies the statement of Theorem
\ref{Theorem12.8} in the case $0<\beta<3/2$.

If $\beta=3/2$, then we rewrite the two last terms in the right-hand side of
(\ref{12.12}) as
\begin{equation}\label{Exlog}
\begin{split}
&(w+n)\log\left[1+\frac{\theta_l\left(e^{t/\sigma_l}-1\right)}{w+n}\right]-w
\log\left[1+\frac{\theta_l}{w}\left(e^{t/\sigma_l}-1\right)\right]\\
&+
\theta_l\left(e^{t/\sigma_l}-1\right)\log\left[1+\frac{\theta_l\left(e^{t/\sigma_l}-1\right)}{w+n}\right]
-\theta_l\left(e^{t/\sigma_l}-1\right)
\log\left[1+\frac{\theta_l}{w}\left(e^{t/\sigma_l}-1\right)\right]\\
\end{split}
\end{equation}
As $n\rightarrow\infty$, the first two terms in the expression above are equal to
$$
\frac{n\theta_l^2\left(e^{t/\sigma_l}-1\right)^2}{2w(w+n)} +o(1),
$$
and the third and the fourth terms  in (\ref{Exlog}) are equal to
$$
-\frac{n\theta_l^2\left(e^{t/\sigma_l}-1\right)^2}{w(w+n)} +o(1).
$$
Therefore, equations (\ref{12.15})-(\ref{LogGeneration})  still hold true
in the case $\beta=3/2$, which completes the proof of Theorem \ref{Theorem12.8}, (a).

(b) Assume that $\beta>3/2$.
We use the formula
\begin{equation}\label{11.4}
\begin{split}
&\mathbb{E}\left(e^{\frac{K_n^{(l)}-\tilde{\Delta}_l}{\tilde{\sigma}_l}t}\right)
\sim e^{-\frac{\tilde{\Delta}_l}{\tilde{\sigma}_l}t}
\frac{\left(\theta_l\left(e^{t/\tilde{\sigma}_l}-1\right)+w+n\right)^{\theta_l\left(e^{t/\tilde{\sigma}_l}-1\right)+w+n-1}}{(w+n)^{w+n-1}}\\
&\times
\frac{w^{w-1}}{\left(\theta_l\left(e^{t/\tilde{\sigma}_l}-1\right)+w\right)^{\theta_l\left(e^{t/\tilde{\sigma}_l}-1\right)+w-1}},
\end{split}
\end{equation}
where
\begin{equation}
\tilde{\Delta}_l=\frac{\theta_l}{w},\;\;\;\tilde{\sigma}_l^2=n\frac{\theta_l}{w}
\left(1-\frac{\theta_l}{w}\right).
\end{equation}
We note that
\begin{equation}
\begin{split}
&\frac{\left[\theta_l\left(e^{t/\tilde{\sigma}_l}-1\right)+w+n\right]^{\theta_l\left(e^{t/\tilde{\sigma}_l}-1\right)+w+n-1}}{\left[\theta_l\left(e^{t/\tilde{\sigma}_l}-1\right)+w\right]^{\theta_l\left(e^{t/\tilde{\sigma}_l}-1\right)+w-1}}\\
&=\left[1+\frac{n}{\theta_l\left(e^{t/\tilde{\sigma}_l}-1\right)+w}\right]^{\theta_l\left(e^{t/\tilde{\sigma}_l}-1\right)+w-1}
\left[\theta_l\left(e^{t/\tilde{\sigma}_l}-1\right)+w+n\right]^n.
\end{split}
\end{equation}
Therefore, we can rewrite equation  (\ref{11.4}) as
\begin{equation}\label{11.5}
\begin{split}
&\mathbb{E}\left(e^{\frac{K_n^{(l)}-\tilde{\Delta}_l}{\tilde{\sigma}_l}t}\right)
\sim e^{-\frac{\tilde{\Delta}_l}{\tilde{\sigma}_l}t}
\left[1+\frac{n}{\theta_l\left(e^{t/\tilde{\sigma}_l}-1\right)+w}\right]^{\theta_l\left(e^{t/\tilde{\sigma}_l}-1\right)+w-1}\\
&\times\left[1+\frac{\theta_l\left(e^{t/\tilde{\sigma}_l}-1\right)}{w+n}\right]^n
\left(1+\frac{n}{w}\right)^{w-1}.
\end{split}
\end{equation}
The crucial observation is that $\tilde{\sigma}_l\sim n^{1/2}$, and that
\begin{equation}
n\log\left[1+\frac{\theta_l\left(e^{t/\tilde{\sigma}_l}-1\right)}{w+n}\right]
=n\frac{\theta_lt}{w\tilde{\sigma}_l}+n\frac{\theta_lt^2}{2w\tilde{\sigma}_l^2}
-\frac{n\theta_l^2t^2}{2w^2\tilde{\sigma}_l^2}+o(1)=\frac{\tilde{\Delta}_l t}{\tilde{\sigma}_l}+\frac{t^2}{2}+o(1),
\end{equation}
as $n\rightarrow\infty$. In addition, it is easy to see that
under the assumption $\beta>3/2$
\begin{equation}
\left[\theta_l\left(e^{t/\tilde{\sigma}_l}-1\right)+w-1\right]\log\left[1+
\frac{n}{\theta_l\left(e^{t/\tilde{\sigma}_l}-1\right)+w}\right]
-(w-1)\log\left[1+\frac{n}{w}\right]=o(1),
\end{equation}
as $n\rightarrow\infty$. Therefore,  (\ref{LogGeneration})
holds true, which implies (\ref{11.3}).
Theorem \ref{Theorem12.8} is proved.
\end{proof}

\section{Representation in terms of random matrices with
Poisson entries}\label{SectionRepresentationPoisson}
Denote by $A(n)$ a random matrix of size $n\times k$ whose entries
$A_j^{(l)}(n)$ (where $j=1,\ldots, n$ and $l=1,\ldots,k$)
are random variables taking values in the set $\left\{0,1,2,\ldots\right\}$
of positive integers. Assume that the joint probability distribution
of $A_j^{(l)}(n)$ is given by the refined Ewens sampling formula
(equation (\ref{RefinedESF})). Then $A_j^{(l)}(n)$ represents the number
of alleles of class $l$ that appear $j$ times in the sample of size $n$.
\begin{prop}\label{PROPOSITION9.4}
Let
\begin{equation}\label{Data999}
\left(
\begin{array}{ccc}
  a_1^{(1)} & \ldots & a_1^{(k)} \\
  \vdots &  &  \\
  a_n^{(1)} & \ldots & a_n^{(k)}
\end{array}
\right),
\end{equation}
be a fixed matrix whose entries are non-negative integers satisfying the condition
\begin{equation}\label{NormalizationEntries}
\sum\limits_{l=1}^k\sum\limits_{j=1}^nja_j^{(l)}=n.
\end{equation}
Then for each $n\geq 2$
\begin{equation}
\begin{split}
&\Prob\left(A(n)=\left(
\begin{array}{ccc}
  a_1^{(1)} & \ldots & a_1^{(k)} \\
  \vdots &  &  \\
  a_n^{(1)} & \ldots & a_n^{(k)}
\end{array}
\right)\right)\\
&=\Prob\left(
\left(\begin{array}{ccc}
  \eta_1(\theta_1) &  \ldots & \eta_1(\theta_k) \\
  \vdots &  &  \\
  \eta_n(\theta_1) & \ldots & \eta_n(\theta_k)
\end{array}
\right)=\left(
\begin{array}{ccc}
  a_1^{(1)} & \ldots & a_1^{(k)} \\
  \vdots &  &  \\
  a_n^{(1)} & \ldots & a_n^{(k)}
\end{array}
\right)\biggl\vert\sum\limits_{l=1}^k\sum\limits_{j=1}^{n}j\eta_j(\theta_l)=n\right).
\end{split}
\end{equation}
Here $\eta_j(\theta_l)$ are the independent Poisson random variables with mean $\theta_l/j$, i.e.
\begin{equation}\label{9.4.0}
\Prob\left\{\eta_j(\theta_l)=p\right\}=e^{-\theta_l/j}\left(\frac{\theta_l}{j}\right)^p\frac{1}{p!},\;\; p=0,1,2,\ldots .
\end{equation}
\end{prop}
\begin{proof}
The fact that (\ref{RefinedESF}) defines a probability distribution on the set of all matrices (\ref{Data999})
satisfying (\ref{NormalizationEntries}) implies
\begin{equation}\label{9.5}
\sum\limits_{b_j^{(l)}:\;\sum\limits_{l=1}^k\sum\limits_{j=1}^njb_j^{(l)}=n}
\prod_{l=1}^{k}\prod\limits_{j=1}^n\frac{\left(\theta_l/j\right)^{b_j^{(l)}}}{(b_j^{(l)})!}=
\frac{\left(\theta_1+\ldots+\theta_k\right)_n}{n!}.
\end{equation}
Note that the left-hand side of (\ref{9.5}) can be understood as
$$
\Prob\left\{\sum\limits_{l=1}^k\sum\limits_{j=1}^nj\eta_j(\theta_l)=n\right\}.
$$
We will show that
\begin{equation}\label{9.6}
\begin{split}
&\Prob\left\{A_j^{(l)}(n)=a_j^{(l)};\; j=1,\ldots,n;\; l=1,\ldots,k\right\}\\
&=
\Prob\left\{\eta_j(\theta_l)=a_j^{(l)};\; j=1,\ldots,n;\;l=1,\ldots,k\biggl\vert
\sum\limits_{l=1}^k\sum\limits_{j=1}^nj\eta_j(\theta_l)=n\right\}.
\end{split}
\end{equation}
Indeed, the right-hand  side of equation (\ref{9.6}) can be computed as
\begin{equation}\label{9.7}
\begin{split}
&
\Prob\left\{\eta_j(\theta_l)=a_j^{(l)};\; j=1,\ldots,n;\;l=1,\ldots,k\biggl\vert
\sum\limits_{l=1}^k\sum\limits_{j=1}^nj\eta_j(\theta_l)=n\right\}\\
&=\frac{\Prob\left\{\eta_j(\theta_l)=a_j^{(l)};\;j=1,\ldots,n;\;l=1,\ldots,k\right\}}{\Prob\left\{\sum\limits_{l=1}^k\sum\limits_{j=1}^nj\eta_j(\theta_l)=n\right\}}\\
&=\frac{n!}{\left(\theta_1+\ldots+\theta_k\right)_n}
\prod\limits_{l=1}^k\prod\limits_{j=1}^n\frac{\left(\theta_l/j\right)^{a_j^{(l)}}}{\left(a_j^{(l)}\right)!}e^{-\theta_l/j},
\end{split}
\end{equation}
which is equal to the probability in the left-hand side of equation (\ref{9.6}).
\end{proof}
Our next result is Theorem \ref{TheoremConvergencePoisson} which says that as $n\rightarrow\infty$ the random matrix $A(n)$ can be approximated by a matrix whose entries are independent Poisson random variables.
\begin{thm}\label{TheoremConvergencePoisson}
For any fixed strictly positive integer $m$, it holds that
\begin{equation}
\left(
\begin{array}{ccc}
  A_1^{(1)}(n) & \ldots & A_1^{(k)}(n) \\
  \vdots &  &  \\
  A_m^{(1)}(n) & \ldots & A_m^{(k)}(n)
\end{array}
\right)\overset{D}{\longrightarrow}
\left(\begin{array}{ccc}
  \eta_1(\theta_1) &  \ldots & \eta_1(\theta_k) \\
  \vdots &  &  \\
  \eta_m(\theta_1) & \ldots & \eta_m(\theta_k)
\end{array}
\right),
\end{equation}
as $n\rightarrow\infty$. Here, $\eta_j(\theta_l)$ are the independent Poisson random variables with mean $\theta_l/j$.
\end{thm}
\begin{proof}
We want to compute the asymptotics of
\begin{equation}\label{12.7.1}
\Prob\left\{A_j^{(l)}(n,\theta_l)=a_j^{(l)},\; j=1,\ldots,m, l=1,\ldots,k\right\}
\end{equation}
as $n\rightarrow\infty$, where $a_j^{(l)}$ are
non-negative integers satisfying condition (\ref{NormalizationEntries}). By Proposition \ref{PROPOSITION9.4} we can
rewrite  (\ref{12.7.1}) as
\begin{equation}\label{12.7.2}
\begin{split}
&\Prob\left\{\eta_j(\theta_l)
=a_j^{(l)},\; j=1,\ldots,m,\; l=1,\ldots,k\right\}\\
&\times\frac{\Prob\left\{\sum\limits_{l=1}^k\sum\limits_{j=m+1}^nj\eta_j(\theta_l)=n-a^{(1)}-\ldots-a^{(k)}\right\}}{\Prob\left\{\sum\limits_{l=1}^k\sum\limits_{j=1}^nj\eta_j(\theta_l)=n\right\}},
\end{split}
\end{equation}
where $a^{(l)}=\sum\limits_{j=1}^mja_j^{(l)}$, $l=1,\ldots,k$.
It is enough to show that
\begin{equation}
\begin{split}
&\frac{\Prob\left\{\sum\limits_{l=1}^k\sum\limits_{j=m+1}^nj\eta_j(\theta_l)=n-a^{(1)}-\ldots-a^{(k)}\right\}}{\Prob\left\{\sum\limits_{l=1}^k\sum\limits_{j=1}^nj\eta_j(\theta_l)=n\right\}}\longrightarrow 1,
\end{split}
\end{equation}
as $n\rightarrow\infty$. Denote
\begin{equation}\label{12.7.3}
T_{m,n}(\theta_l)=\sum\limits_{j=m+1}^nj\eta_j(\theta_l).
\end{equation}
The right-hand side of (\ref{12.7.2}) has the form
\begin{equation}\label{12.7.4}
\frac{\Prob\left\{T_{m,n}(\theta_1)+\ldots+T_{m,n}(\theta_k)=n-a^{(1)}-\ldots-a^{(k)}\right\}}{\Prob\left\{T_{0,n}(\theta_1)+\ldots+T_{0,n}(\theta_k)=n\right\}}.
\end{equation}
Formula (\ref{9.5}) implies
\begin{equation}
\Prob\left\{T_{0,n}(\theta_1)+\ldots+T_{0,n}(\theta_k)=n\right\}=
\frac{(\theta_1+\ldots+\theta_k)_n}{n!}e^{-(\theta_1+\ldots+\theta_k)\sum_{j=1}^nj^{-1}}.
\end{equation}
In order to proceed, we wish to compute the generating function of the random variable $\sum\limits_{l=1}^kT_{m,n}(\theta_l)$. Using equation (\ref{12.7.3})
we obtain
\begin{equation}\label{9T}
\begin{split}
&\mathbb{E}\left(x^{\sum\limits_{l=1}^kT_{m,n}(\theta_l)}\right)
=\mathbb{E}\left(x^{\sum\limits_{l=1}^k
\sum\limits_{j=m+1}^nj\eta_j(\theta_l)}\right)
=\prod\limits_{l=1}^k\prod\limits_{j=m+1}^n
\mathbb{E}\left(x^{j\eta_j(\theta_l)}\right)\\
&=\prod\limits_{l=1}^k\prod\limits_{j=m+1}^n
e^{\frac{\theta_l}{j}(x^j-1)}.
\end{split}
\end{equation}
For a positive integer-valued random variable $Z$ the following formula holds true
\begin{equation}\label{10T}
\Prob\left(Z=k\right)=\frac{1}{k!}\left(\frac{d^k}{dx^k}\mathbb{E}\left(x^Z\right)\right)\biggl\vert_{x=0}.
\end{equation}
Using (\ref{9T}) and (\ref{10T}) we obtain the formula
\begin{equation}\label{11.1T}
\begin{split}
&(n-a^{(1)}-\ldots-a^{(k)})!
\Prob\left(\sum\limits_{l=1}^kT_{m,n}(\theta_l)=n-a^{(1)}-\ldots-a^{(k)}\right)\\
&=e^{-\sum\limits_{l=1}^k\sum\limits_{j=m+1}^n\frac{\theta_l}{j}}\biggl[\frac{d^{n-a^{(1)}-\ldots-a^{(k)}}}{dx^{n-a^{(1)}-\ldots-a^{(k)}}}
e^{\sum\limits_{l=1}^k\sum\limits_{j=m+1}^n\frac{\theta_lx^j}{j}}\biggr]\biggr\vert_{x=0}.
\end{split}
\end{equation}
Set
\begin{equation}
\begin{split}
&g_l(x)=e^{-\theta_l\sum\limits_{j=1}^{m}
\frac{x^j}{j}},\; l=1,\ldots,k;\;\; G(x)=g_1(x)\ldots g_k(x).
\end{split}
\end{equation}
Then
\begin{equation}\label{12.1T}
\begin{split}
&
\Prob\left(\sum\limits_{l=1}^kT_{m,n}(\theta_l)=n-a^{(1)}-\ldots-a^{(k)}\right)\\
&=\frac{e^{-\sum\limits_{l=1}^k\sum\limits_{j=m+1}^n\frac{\theta_l}{j}}}{(n-a^{(1)}-\ldots-a^{(k)})!}
\biggl[\frac{d^{n-a^{(1)}-\ldots-a^{(k)}}}{dx^{n-a^{(1)}-\ldots-a^{(k)}}}
\left(G(x)e^{-\left(\theta_1+\ldots+\theta_k\right)\log(1-x)}\right)
\biggr]\biggr\vert_{x=0}.
\end{split}
\end{equation}
If $A=a^{(1)}+\ldots+a^{(k)}$, and $\Omega=\theta_1+\ldots+\theta_k$, then
\begin{equation}
\begin{split}
&\biggl[\frac{d^{n-A}}{dx^{n-A}}\left(G(x)e^{-\Omega\log(1-x)}\right)\biggr]\biggl\vert_{x=0}
=G(1)\left(\Omega\right)_{n-A}-G'(1)(\Omega-1)_{n-A}+\frac{G^{''}(1)}{2}(\Omega-2)_{n-A}-\ldots\\
&=e^{-\sum\limits_{l=1}^k\sum\limits_{j=1}^{m}\theta_l}\left(\Omega\right)_{n-A}\left[1+\frac{\left(m\sum\limits_{l=1}^k\theta_l\right)(\Omega-1)}{\Omega-1+n-A}+O\left(\frac{1}{n^2}\right)\right].
\end{split}
\nonumber
\end{equation}
Therefore,
\begin{equation}\label{13.1T}
\begin{split}
&
\Prob\left(\sum\limits_{l=1}^kT_{m,n}(\theta_l)=n-a^{(1)}-\ldots-a^{(k)}\right)=\frac{e^{-\sum\limits_{l=1}^k\sum\limits_{j=1}^n\frac{\theta_l}{j}}
\left(\theta_1+\ldots+\theta_k\right)_{n-a^{(1)}-\ldots-a^{(k)}}}{(n-a^{(1)}-\ldots-a^{(k)})!}\\
&\times
\left[1+\frac{\left(m\sum\limits_{l=1}^k\theta_l\right)(\theta_1+\ldots+\theta_k-1)}{\theta_1+\ldots+\theta_k-1+n-a^{(1)}-\ldots-a^{(k)}}+O\left(\frac{1}{n^2}\right)\right].
\end{split}
\end{equation}
We conclude that
\begin{equation}
\begin{split}
&\frac{\Prob\left\{\sum\limits_{l=1}^k\sum\limits_{j=m+1}^nj\eta_j(\theta_l)=n-a^{(1)}-\ldots-a^{(k)}\right\}}{\Prob\left\{\sum\limits_{l=1}^k\sum\limits_{j=1}^nj\eta_j(\theta_l)=n\right\}}\\
&=\frac{n!}{(n-a^{(1)}-\ldots-a^{(k)})!}\;
\frac{\left(\theta_1+\ldots+\theta_k\right)_{n-a^{(1)}-\ldots-a^{(k)}}}{\left(\theta_1+\ldots+\theta_k\right)_{n}}\\
&\times
\left[1+\frac{\left(m\sum\limits_{l=1}^k\theta_l\right)(\theta_1+\ldots+\theta_k-1)}{\theta_1+\ldots+\theta_k-1+n-a^{(1)}-\ldots-a^{(k)}}+O\left(\frac{1}{n^2}\right)\right]\longrightarrow 1,
\end{split}
\end{equation}
as $n\rightarrow\infty$. Theorem \ref{TheoremConvergencePoisson} is proved.
\end{proof}
\begin{rem}In the situation where the mutation parameter grows with the sample size $n$ the Poisson approximation  for the Ewens sampling formula
was considered in Tsukuda \cite{Tsukuda1}. Methods of Ref. \cite{Tsukuda1}
can be applied to the refined Ewens sampling formula (equation (\ref{RefinedESF}))
as well.
We postpone this issue to future work.
\end{rem}



\begin{thebibliography}{99}
\bibitem{Aldous}Aldous, D. J. Exchangeability and related topics. $\acute{\mbox{E}}$cole d'$\acute{\mbox{e}}$t$\acute{\mbox{e}}$ de probabilit$\acute{\mbox{e}}$s de Saint-Flour, XIII—1983, 1–198, Lecture Notes in Math., 1117, Springer, Berlin, 1985.
\bibitem{Antoniak} Antoniak, C. E. Mixtures of Dirichlet processes with applications to Bayesian nonparametric problems. Ann. Statist. 2 (1974), 1152–-1174.
\bibitem{ArratiaBarbourTavarePoisson}
Arratia, R.; Barbour, A.; Tavar$\acute{\mbox{e}}$, S.
Poisson process approximation for the Ewens sampling formula. Ann. Appl. Probab. 2 (1992), no. 3, 519-–535.
\bibitem{ArratiaBarbourTavare} Arratia, R.; Barbour, A.; Tavar$\acute{\mbox{e}}$, S. Logarithmic combinatorial structures: a probabilistic approach. European Mathematical Society, Zurich, 2003.

\bibitem{Crane} Crane, H. The Ubiquitous Ewens Sampling
Formula. Statistical Science 31 (2016) 1--19.

\bibitem{Durrett} Durrett, R. Probability Models for DNA Sequence Evolution.
Probability and Its Applications. Springer 2008.

\bibitem{Etheridge} Etheridge, A. Some Mathematical Models from Population Genetics. Lecture Notes in Mathematics. Springer 2012.

\bibitem{Ewens}Ewens, W. J. The sampling theory of selectively neutral alleles. Theoret. Population Biol. 3 (1972).

\bibitem{EwensLectures} Ewens, W. J. Mathematical Population Genetics.
Lecture Notes. Cornell University 2006. Available online
https://services.math.duke.edu/~rtd/CPSS2006/cornelllect.pdf


\bibitem{FengL} Feng, S. Large deviations associated with Poisson-Dirichlet distribution and Ewens sampling formula.
Ann. Appl. Probab. 17 (2007), no. 5--6, 1570–-1595.

\bibitem{Feng}
Feng, S. The Poisson-Dirichlet Distribution and Related Topics. Models and Asymptotic  Behaviors.
Springer 2010.

\bibitem{GhosalvanderVaart} Ghosal, S.; van der Vaart, A. Fundamentals of nonparametric Bayesian inference. Cambridge Series in Statistical and Probabilistic Mathematics, 44. Cambridge University Press, Cambridge, 2017.

\bibitem{Gnedin} Gnedin, A. V. Three sampling formulas. Combin. Probab. Comput. 13 (2004), no. 2, 185–-193.

\bibitem{GriffithsLessard}
Griffiths, R. C.; Lessard, S. Ewens' sampling formula and related formulae: combinatorial proofs, extension to a variable population size and applications to ages of alleles. Theor. Population Biology
68 (2005) 167--177.

\bibitem{HongPark} Hong, E. P.; Park, J. W.
Sample Size and Statistical Power Calculation in Genetic Association Studies.
Genomics Inf. (2012); 10(2): 117–-122.

\bibitem{Hoppe1}Hoppe, F. M. Pólya-like urns and the Ewens' sampling formula. J. Math. Biol. 20 (1984), no. 1, 91–-94.

\bibitem{Hoppe2} Hoppe, F. M. The sampling theory of neutral alleles and an urn model in population genetics. J. Math. Biol. 25 (1987), no. 2, 123–-159.









\bibitem{Kingman1} Kingman, J. F. C. Random partitions in population genetics. Proc. R. Soc. London A 361, (1978), 1--20.
\bibitem{Kingman2}Kingman, J. F. C. The representation of partition structures. J. London Math. Soc. 18 (1978), 374--380.
\bibitem{KorwarHollander}  Korwar, R. M.; Hollander, M. Contributions to the theory of Dirichlet processes. Ann. Probability 1 (1973), 705–-711.
\bibitem{KotzBalakrishnanJohnson} Kotz, S.; Balakrishnan, N.; Johnson, N. L. Continuous multivariate distributions. Vol. 1. Models and applications. Second edition. Wiley Series in Probability and Statistics: Applied Probability and Statistics.
\bibitem{Macdonald} Macdonald, I. Symmetric Functions and Hall Polynomials. Second edition,
Oxford Mathematical Monographs, Oxford University Press (1995).
\bibitem{Mahmoud} Mahmoud, H. M. P$\acute{\mbox{o}}$lya Urn Models. Texts in Statistical Sciences (2009).
\bibitem{Pitman1} Pitman, J. The two-parameter generalization of Ewens' random partition structure.
Technical Report 345, Dept Statistics, University of California, Berkeley (1992).
\bibitem{Pitman2} Pitman, J. Exchangeable and partially exchangeable random partitions. Probab. Theory Related Fields 102 (1995), no. 2, 145–-158.

\bibitem{PitmanBookCombinatorialProcesses}
Pitman, J. Combinatorial Stochastic Processes.
Ecole d'Et$\acute{\mbox{e}}$ de Probabilit$\acute{\mbox{e}}$s
de  Saint-Flour XXXII-2002.
Lecture Notes in Mathematics 1875,
Spinger 2002.

\bibitem{ReillyElphick} O'Reilly, R.; Elphick, H. E. Development, clinical utility, and place of ivacaftor in the treatment of cystic fibrosis. Drag  Design, Development and Therapy 7 (2013), 929--937.

\bibitem{StrahovMPS}
 Strahov, E. Multiple partition structures and harmonic functions on branching graphs. Adv. in Appl. Math. 153 (2024), Paper No. 102617, 49 pp.

\bibitem{StrahovGRR}
Strahov, E. Generalized regular representations of big wreath products. Israel J. of Math. (2025)
(will appear)

\bibitem{TavareBook} Tavar$\acute{\mbox{e}}$, S.
Ancestral Inference in Population Genetics.
Lectures on Probability Theory and Statistics. Ecole d'Et$\acute{\mbox{e}}$
de Probabilit$\acute{\mbox{e}}$s de Saint-Flour XXXI (2001).

\bibitem{Tavare} Tavar$\acute{\mbox{e}}$, S.  The magical Ewens sampling formula.
Bull. London Math. Soc. 53 (2021), 1563--1582.

\bibitem{Tsukuda}
Tsukuda, K. Estimating the large mutation parameter of the Ewens sampling formula. J. Appl. Probability 54 (2017), 42--54.


\bibitem{Tsukuda1}
 Tsukuda, K.  On Poisson approximations for the Ewens sampling formula when the mutation parameter grows with the sample size. Ann. Appl. Probab. 29 (2019), no. 2, 1188–-1232.
 \bibitem{Watterson}
Watterson, G. A. The stationary distribution of the infinitely-many neutral alleles diffusion model. J. Appl. Probability 13 (1976), no. 4, 639–-651.

\end{thebibliography}
\end{document}